\documentclass[11pt]{amsart}
\allowdisplaybreaks[1]

\usepackage{amsmath}
\usepackage{amsxtra}
\usepackage{amssymb}
\usepackage{amscd}
\usepackage{epsfig}
\usepackage{comment}

\usepackage{mathrsfs}
\usepackage{color}

\def\wtil{\widetilde}

 \newcommand{\N}{\mathbb N}
 \newcommand{\bee}{\begin{equation}}
 \newcommand{\eee}{\end{equation}}
 
 \newcommand{\Lb}{\mbox {\boldmath ${\Lambda}$}}
 \newcommand{\Gb}{\mbox {\boldmath ${\Gamma}$}}

 \newcommand{\Lbs}{\mbox{\scriptsize\boldmath ${\Lambda}$}}
 \newcommand{\Lbt}{\mbox{\tiny\boldmath ${\Lambda}$}}
 
 \newcommand{\Gbb}{\mbox {\bf P}}
 \newcommand{\Gbbs}{\mbox {\scriptsize{\bf P}}}
 
 \def\Gbb{\mbox{\bf G}}
 \def\Gbbs{\mbox {\scriptsize{\bf G}}}
 
 \newcommand{\Qb}{\mbox {\bf Q}}

 \def\Chi{\mbox{\large${\chi}$}}

\newcommand{\diam}{\mbox{\rm diam}}
\newcommand{\be}{\begin{eqnarray}}
\newcommand{\ee}{\end{eqnarray}}
\newcommand{\supp}{\mbox{\rm supp}}

\newcommand{\freq}{\mbox{\rm freq}}
\newcommand{\Vol}{\mbox{\rm Vol}}

\newcommand{\eps}{{\mbox{$\epsilon$}}}

\newcommand{\R}{{\mathbb R}}
\newcommand{\Q}{{\mathbb Q}}
\newcommand{\Z}{{\mathbb Z}}
\newcommand{\C}{{\mathbb C}}

\newcommand{\Ak}{{\mathcal A}}

\newcommand{\Sf}{{\mbox{\sf S}}}
\newcommand{\Dk}{{\mathcal D}}

\newcommand{\Pk}{{\mathcal P}}

\newcommand{\Int}{{\rm int}}

\newcommand{\Ck}{{\mathcal C}}

\newcommand{\Kk}{{\mathcal K}}

\newcommand{\Vk}{{\mathcal V}}

\newcommand{\Sk}{{\mathcal S}}
\newcommand{\Tk}{{\mathcal T}}
\newcommand{\Xk}{{\mathcal X}}

\newcommand{\dist}{\mbox{\rm dist}}

\newcommand{\Lam}{{\Lambda}}

\newcommand{\lam}{\lambda}

\newcommand{\om}{\omega}

\def\sp{{\rm sp}}

\newcommand{\balpha}{\mbox{\boldmath{$\alpha$}}}

 \newtheorem{theorem}{Theorem}[section]
 \newtheorem{lemma}[theorem]{Lemma}
 \newtheorem{prop}[theorem]{Proposition}
 \newtheorem{cor}[theorem]{Corollary}

 \newtheorem{defi}[theorem]{Definition}
 \newtheorem{example}[theorem]{Example}
 \newtheorem{remark}[theorem]{Remark}

\numberwithin{equation}{section}

\begin{document}

\title[Substitution tilings without FLC]{On substitution tilings and Delone sets without finite local complexity}


\hspace*{12pt}
\email{jylee@cku.ac.kr}
\email{bsolom3@gmail.com}

\date{\today}


\maketitle

\centerline{Jeong-Yup Lee $^{\,\rm a}$ and Boris Solomyak $^{\,\rm b}$}

\hspace*{4em}

{\footnotesize
 \hspace*{4em} a: Department of Mathematics Education, Catholic Kwandong University, 
\\ \hspace*{4em} \hspace*{2.5em} Gangneung, Gangwon 210-701, Korea

\smallskip

\hspace*{4em} b: Department of Mathematics, Bar-Ilan University, Ramat-Gan 52900,
\\ \hspace*{4em} \hspace*{2.5em}   Israel}

\begin{abstract}
We consider substitution tilings and Delone sets without the assumption of finite local complexity (FLC). We first give a sufficient condition for tiling dynamical systems to be uniquely ergodic and a formula for the measure of cylinder sets. We then obtain several results on their ergodic-theoretic properties, notably absence of strong mixing and conditions for existence of eigenvalues, which have number-theoretic consequences. In particular, if the set of eigenvalues of the expansion matrix is totally non-Pisot, then the tiling dynamical system is weakly mixing. Further, we define the notion of rigidity for substitution tilings and demonstrate that the result of \cite{LeeSol:12} on the equivalence of four properties: 
relatively dense discrete spectrum, being not weakly mixing, the Pisot family, and the Meyer set property, extends to the non-FLC case, if we assume rigidity instead.
\end{abstract}


\section{Introduction}

Delone sets and tilings of the Euclidean space $\R^d$  have been used for a long time to model physical structures. It is often fruitful, moreover, to consider not just an individual object (tiling or Delone set), but rather its ``hull,'' defined as a closure of the collection of all its translates in a natural topology. The group $\R^d$ acts on the hull by translations, and we obtain a dynamical system, whose properties have a direct connection with the physical properties of the material. Some of the references to this approach are \cite{RW,Hof,BHZ,Len-Sto}, see also the recent book \cite{BaakeGrimm}.
One  of the central issues for the resulting dynamical system is whether it has {\em unique ergodicity}, which allows one to use powerful ergodic theorems that are important for applications.
Another much studied aspect of a system is its {\em spectrum}, and it has been realized for a long time that there is a close link between the ``dynamical'' and the ``diffraction'' spectrum. The list of references here is too long, so we refer the reader to the recent survey \cite{BaaLen17} and its bibliography. 

Here we  investigate dynamical systems associated with substitution tilings and Delone sets in $\R^d$, without the assumption of finite local complexity (FLC). Whereas the theory of FLC tilings is well-developed, much less is known about non-FLC tilings, which are sometimes called ILC (for ``infinite local complexity''). We consider primitive substitution tilings, with finitely many prototiles up to translation. In this class, the ILC phenomenon (at least, in the planar case) is generally caused  by tilings with tiles whose edges are ``sliding'' (in the literature, this phenomenon has been referred to as an ``earthquake'' or a ``fault line''). Thus our setting does not cover pinwheel type tilings in which prototiles appear with infinitely many rotations. For dynamical properties of the latter, we refer the reader to \cite{Radin, Fre-Ric}.
First examples of ILC substitution tilings were constructed by Danzer \cite{Danzer} and Kenyon \cite{Kenyon92}, and a large class of such tilings was studied by Frank and Robinson \cite{FraRob}. More recently, investigation of ILC tilings appeared in the work of Frank and  Sadun
\cite{Fra-Sa,FraSa2}, see also \cite{Frank}, in a general framework of ``fusion tilings''. In particular, they established unique ergodicity under some assumptions and obtained a number of results on the topological and ergodic-theoretic properties of ILC tilings.

In this paper, we revisit the theory developed in \cite{soltil,LMS2,sol-eigen,LeeSol:08,LeeSol:12} and extend parts of it to include the ILC case.
We work in parallel in the frameworks of Delone $\kappa$-sets and tilings,  which are essentially equivalent. The notion of Delone $\kappa$-set formalizes the idea of a ``coloured Delone set,'' in which to each point a particular colour is assigned, from a finite list. This way we can model structures with atoms of different kind. The first issue that we address is unique ergodicity.
In the FLC case unique ergodicity is equivalent to the existence of {\em uniform cluster frequencies} (UCF) \cite{LMS1}. In the ILC case the UCF property is no longer sufficient; however, we show that under an additional technical condition, the unique ergodicity follows, and this is our first main result, Theorem~\ref{th-UE2}. Although  this theorem is general, it is adapted for substitution tilings with finitely many prototiles up to translation, but does not apply, for instance, to the pinwheel tiling. Along the way we also obtain a formula for the measure of cylinder sets in terms of patch frequencies.

   In the FLC case the property of ``linear repetitivity'' is known to imply unique ergodicity, see \cite[Thm 6.1]{LP}, \cite[Thm 2.7]{LMS1}, and \cite[Cor 4.6]{DL}.  Frettl\"{o}h and Richard \cite{Fre-Ric} developed versions of this property suitable for the ILC case. The two main ones, the ``linear wiggle-repetitivity'' and ``almost linear repetitivity'' were shown in \cite{Fre-Ric} to imply unique ergodicity. The former property is satisfied, in particular, for the pinwheel tiling. It is an open question whether for tilings with fault lines, like in \cite{FraRob}, the corresponding point sets are almost linearly repetitive. We should note that the results of Frank and Sadun \cite{FraSa2} on unique ergodicity are more general and cover substitution systems of both kinds; however, the framework that we develop is better suited for our purposes in the sequel.

Next we turn to ILC substitution systems under our assumptions, and to their ergodic-theoretic and spectral properties.
We prove that these systems are not strongly mixing and  obtain a necessary condition for eigenvalues, generalizing results of \cite{soltil,sol-eigen}. If we also assume repetitivity and recognizability, as well as an algebraic property of the expansion map, this condition becomes sufficient, and all eigenfunctions may be chosen continuous. As in the FLC case, this leads to number-theoretic considerations.
If the expansion is a pure dilation, then the existence of non-trivial eigenfunctions (equivalently, absence of weak mixing) implies that the expansion constant is a Pisot (or PV) number. More generally, if the set of eigenvalues of the expansion map $Q$ is ``totally non-Pisot'' \cite{Robi.lec}, then the tiling dynamical system is weakly mixing.
Interestingly, if the set of eigenvalues {of the dynamical system} form a relatively dense subset of  $\R^d$, this forces the set of translation vectors between equivalent tiles to be Meyer, and this implies FLC. This phenomenon is also related to an important property of the tiling which we call ``rigidity''. Rigidity was  established  in \cite{LeeSol:12} for FLC primitive substitution tilings, under some assumptions of algebraic nature.
 Actually, rigidity is easy to check directly in examples, and it holds for many ILC tilings, specifically, for those considered in \cite{FraRob}. However, it does not hold for Kenyon's tiling \cite{Kenyon92} and its relatives. For these tilings, the dynamical system may have non-trivial discrete spectrum; however, it will not be relatively dense.
We demonstrate that the result of \cite{LeeSol:12} on the equivalence of four properties: 
relatively dense discrete spectrum, being not weakly mixing, the Pisot family, and the Meyer set property, extends to the non-FLC case, if we assume rigidity instead.

Let us comment on the structure of the paper and on the issue of Delone $\kappa$-sets vs.\ tilings. Section \ref{preliminary} is devoted to preliminaries.
In section \ref{unique-ergodicity}, we introduce a special class of patches and cylinder sets which generate topology and the Borel $\sigma$-algebra of the space under consideration.
We then provide sufficient conditions for unique ergodicity in terms of frequencies of these patches (Theorem \ref{th-UE2} and Corollary \ref{cor-tech}). All this section, which is of central importance for us, is entirely in terms of Delone $\kappa$-sets. Starting from Section \ref{substitution-tilings} we work in parallel with substitution tilings and (representable) substitution Delone $\kappa$-sets, which are, in a sense, dual to each other, and we think that both viewpoints are important. The switch between these two settings is made by the results of \cite{lawa, LMS2}. 
We show that the sufficient condition for unique ergodicity holds for ILC primitive substitution tilings,  so it is valid for representable primitive substitution Delone $\kappa$-sets as well.
In section \ref{relDenseEigenvalue-PisotFamily-Meyerset}, we prove  Lemma \ref{intersection-of-cylinderSet} which implies absence of strong mixing and provides the key tool for the characterization of dynamical eigenvalues in Theorem \ref{eigen-thm}. 
In section \ref{section:Rigidity} we define the rigidity and consider its implications. In particular, we give an answer to a question of  Lagarias from \cite{Lag00} in a more general setting than in \cite{LeeSol:08} (without the FLC assumption), namely,  a primitive substitution tiling which is pure point diffractive must have the Meyer property. We also include several examples.

\section{Preliminaries} \label{preliminary}

\subsection{Delone $\kappa$-sets}

A {\em $\kappa$-set} or $\kappa$-coloured set in $\R^d$ is a subset $\Lb$ of $\R^d\times \{1,\ldots,\kappa\}$ whose projection on $\R^d$ is 1-to-1. Here $\kappa$ is a number of colours, and the set
$\Lam_i=\{x\in \R^d\,|\,(x,i)\in \Lb\}$ has colour $i$. It is convenient to write $\Lb = (\Lam_1, \dots, \Lam_{\kappa}) = (\Lam_i)_{i\le \kappa}$ and think about it as a ``vector of sets''.
Recall that a Delone set is a relatively dense and uniformly discrete subset of $\R^d$.
We say that $\Lb=(\Lambda_i)_{i\le \kappa}$ is a {\em Delone $\kappa$-set} \footnote{In the literature \cite{lawa}, there is also the notion of {\em Delone multisets}; they differ from the Delone $\kappa$-sets. In a Delone multiset, points with the same colour may be counted with multiplicity.} in $\R^d$ if
each $\Lambda_i$ is Delone and $\supp(\Lb):=\bigcup_{i=1}^{\kappa} \Lambda_i \subset \R^d$ is Delone. 
A {\em cluster} of $\Lb$ is, by definition,
a family $\Gbb = (G_i)_{i\le \kappa}$ where $G_i \subset \Lambda_i$ is
finite for all $i\le \kappa$.
Many of the clusters that we consider have the form
$ \Lb \cap A := (\Lambda_i \cap A )_{i\le \kappa}$, for a bounded set
 $A\subset \R^d$.
The translate of a cluster $\mbox{\bf G}$ by $x \in \R^d$ is
$x + \Gbb = (x + G_i)_{i\le \kappa}$.
We say that two clusters $\Gbb$ and $\Gbb'$ are {\em translationally equivalent}
if $\Gbb=x+\Gbb'$ for some $x \in \R^d$.

We say that a Delone $\kappa$-set $\Lb$ has {\em finite local complexity(FLC)} if for every $R > 0$ there exists a finite set $Y \subset \supp(\Lb) = \bigcup_{i=1}^{\kappa} \Lam_i$ such that
for all $x \in \supp(\Lb)$, there exists $y \in Y$ for which $B_R(x) \cap \Lb = (B_R(y) \cap \Lb) + (x - y)$.

The Delone $\kappa$-set $\Lb$ is called {\em repetitive} if every $\Lb$-cluster occurs relatively dense in space, up to translation. More precisely,  this means that for any cluster
$\Gbb\subset \Lb$ there exists $M>0$ such that every ball of radius $M=M(\Gbb)$ contains a translated copy of $\Gbb$. (In some papers, see e.g.\ \cite{Fre-Ric}, this property is called
{\em weak repetitivity}, to distinguish it from the condition that  for every $R>0$ there exists $M=M(R)$ such that
{every ball of radius $M$} contains a translated copy of {\em every} $\Lb$-cluster of radius $R$. However, the latter property can only hold in the FLC case, and then the two repetitivity properties are equivalent.)

We will make use of the following notation:
\be \label{nota1}
F^{+r} := \{x\in \R^d:\,\dist(x,F)\le r\}\ \ \mbox{and}\ \ F^{-r}:= \{x\in F:\ \dist(x,\partial F) \ge r\}.
\ee

A {\em van Hove sequence} for $\R^d$ is a sequence
$\mathcal{F}=\{F_n\}_{n \ge 1}$ of bounded measurable subsets of
$\R^d$ satisfying
\be \label{Hove}
\lim_{n\to\infty} \Vol((\partial F_n)^{+r})/\Vol(F_n) = 0,~
\mbox{for all}~ r>0.
\ee

Suppose that a Delone $\kappa$-set $\Lb$ is given.  For a cluster $\Gbb$ of $\Lb$ and a bounded set $F\subset \R^d$, we denote
$$
L_{\Gbbs}(F,\Lb) = \# \{g\in \R^d:\ g+ \Gbb \subset \Lb,\ (g+\supp(\Gbb))\subset F\},
$$
where the symbol $\#$ stands for cardinality of a set. We sometimes write $L_{\Gbbs}(F,\Lb) = L_{\Gbbs}(F)$ when $\Lb$ is understood from the context.

\begin{defi} \label{def-ucf}
{\em Let $\{F_n\}_{n \ge 1}$ be a van Hove sequence.
The Delone $\kappa$-set $\Lb$ has {\em uniform cluster frequencies} (UCF)
(relative to $\{F_n\}_{n \ge 1}$) if for any cluster $\Gbb$,  the limit
$$
\freq(\Gbb,\Lb) = \lim_{n\to \infty} \frac{L_{\Gbbs}(x+F_n, \Lb)}{\Vol(F_n)} \ge 0,
$$
exists uniformly in $x\in \R^d$.}
\end{defi}

We will use the following notation:
for any Delone $\kappa$-set $\Gb = (\Gamma_i)_{i \le \kappa} $ and $B \subset \R^d$, we let
\[  \Gb + B := \{(x+t, i) \ | \ x \in \Gamma_i, t \in B, 1 \le i \le \kappa \},  \]
\[  \Gb + t = \{ (x+t, i) \ | \ x \in \Gamma_i, 1 \le i \le \kappa \},  \ \mbox{where $t \in \R^d $}, \]
\[  \Gb \textcircled{+} B := \{ \Gb + t \ | \ t \in B \}.  \] 


\subsection{Tilings}

We begin with a set of types (or colours) $\{1,\ldots, \kappa \}$, which
we fix once and for all. A {\em prototile} in $\R^d$ is defined as a
pair $T_i=(A_i,i)$ where $A_i=\supp(T_i)$ (the support of $T_i$) is a
compact set in $\R^d$, which is the closure of its interior, and
$i=l(T_i)\in \{1,\ldots, \kappa \}$ is the type of $T$. 
{Note that we do not assume here that the prototiles are connected.} 
A translate of a prototile $T_i$ by vector $g\in \R^d$ is $g+T_i = (g+A_i,i)$; it will be called a {\em tile of type} $i$.
A {\em tiling} of $\R^d$ is a set $\Tk$ of tiles such that $\R^d =
\bigcup \{\supp(T) : T \in \Tk\}$ and distinct tiles have disjoint
interiors. (Strictly speaking, we should say ``disjoint interiors of the support,'' but we sometimes allow this sloppiness of language, when it does not lead to a confusion.)

We say that a set $P$ of tiles is a {\em
patch} if the number of tiles in $P$ is finite and the tiles of
$P$ have mutually disjoint interiors. The {\em support of a patch}
is the union of the supports of the tiles that are in it. The {\em
translate of a patch} $P$ by $g\in \R^d$ is $g+P := \{g+T:\ T\in
P\}$. We say that two patches $P_1$ and $P_2$ are {\em
translationally equivalent} if $P_2 = g+P_1$ for some $g\in \R^d$.
Given a tiling $\Tk$, a finite set of tiles of $\Tk$ is
called a $\Tk$-patch. 

We define FLC and repetitivity for  tilings in the same way as the corresponding properties for  Delone  $\kappa$-sets.
The types (or colours) of tiles for tilings have the meaning analogous
to the colours of points for Delone $\kappa$-sets. We
always assume that any two $\Tk$-tiles of the same type are
translationally equivalent (hence there are finitely many
$\Tk$-tiles up to translations). For a patch $P$ and $F\subset \R^d$ we denote
$$
L_P(F,\Tk) = \#\{g\in \R^d:\ g + P \subset \Tk,\ (g + \supp(P)) \subset F\},
$$
and we sometimes write $L_P(F) = L_P(F,\Tk)$ when $\Tk$ is understood from the context.
Given a van Hove sequence $\{F_n\}_{n\ge 1}$, we define the property of uniform patch frequencies for a tiling $\Tk$, by analogy with uniform cluster frequencies (UCF) for Delone $\kappa$-sets.

\subsection{From a tiling to Delone $\kappa$-set} Given a tiling $\Tk$ of $\R^d$ with the prototile set $\Ak = \{T_1,\ldots,T_\kappa\}$, we can represent the tiling as follows:
$$
\Tk = \bigcup_{i=1}^\kappa (T_i + \Lam_i),
$$
where $\Lam_i$ is the set of all translates of $T_i$ which appear in $\Tk$. It is clear that every $\Lam_i$ is a Delone set. In order to make sure that $\bigcup_{i=1}^\kappa \Lam_i$ is also Delone, we can assume that every prototile contains the origin in the interior of the support. This way we obtain a Delone $\kappa$-set $\Lb_\Tk: = (\Lam_i)_{i\le \kappa}$, which will be called the {\em Delone $\kappa$-set associated to the tiling $\Tk$}.

Going from a general Delone $\kappa$-set $\Lb$ to a tiling is less obvious in the ILC setting. For instance, if we use the Voronoi tesselation corresponding to the Delone set $\supp(\Lb)$, there is no guarantee that we will get a tiling with a finite set of tiles up to translation. In Section \ref{substitution-tilings} we will discuss what can be done in the special case of substitution Delone $\kappa$-sets.

\subsection{Dynamical systems from point sets and tilings} \label{dynamical-system}

Let $\Lb = (\Lambda_i)_{i \le \kappa}$ be a Delone $\kappa$-set in $\R^d$ and let $X$ be the collection of all Delone $\kappa$-sets.
We equip $X$ with a metric, which comes from the ``local rubber topology'' (see \cite{Mu-Ric, Len-Sto, Baa-Len04} and references therein). This is, essentially, the 
 ``local'' topology used already in \cite{Fell}. 
Formally, this metric is defined as follows:
For $\Lb_1, \Lb_2 \in X$, 
\[ d(\Lb_1, \Lb_2): = \mbox{min}\{\widetilde{d}(\Lb_1, \Lb_2), 2^{-\frac{1}{2}}\},\]
where
$$
\widetilde{d}(\Lb_1, \Lb_2) = \mbox{inf} \Bigl\{ \epsilon>0\ | \ B_{\frac{1}{\epsilon}}(0) \cap \Lb_2 \subset \Lb_1 + B_{\epsilon}(0) \  \mbox{and}
  \ B_{\frac{1}{\epsilon}}(0) \cap \Lb_1  \subset \Lb_2 + B_{\epsilon}(0)  \Bigr\}. 
$$
It is a standard fact that this is indeed a metric. For the reader's convenience, we include the proof in the Appendix.

We consider the topology induced by this metric. 
For any $g \in \R^d$ consider the translation ${\sf T}_g: \Gb \mapsto \Gb -g$ on $X$. This defines a continuous action of the group $\R^d$ on $X$. The set $\{\Lb -g \ | \ g \in \R^d \}$ is the orbit of $\Lb$. 
Let $X_{\Lbs}$ be the orbit closure of $\Lb$. Then one can check that $X_{\Lbs}$ is compact in this metric (although $X$ itself is not!). The system $(X_{\Lbs}, \R^d)$ is called the point set dynamical system associated with $\Lb$.

Similarly, given a tiling $\Tk$, we  can consider the space of all tilings with a fixed finite number of prototiles,  with the metric induced from the metric on the associated Delone sets. Alternatively, an equivalent metric can be defined directly, using the Hausdorff distance between the boundaries of large patches around the origin, see e.g., \cite{Fra-Sa}. The  closure of the orbit $\{\Tk-g\ |\ g\in \R^d\}$ is compact; it will be denoted $X_\Tk$, and the translation $\R^d$-action $(X_\Tk,\R^d)$ will be called the tiling dynamical system associated with $\Tk$. The following is immediate from the definition of the metric and from the assumption that we have a finite number of prototiles, but it is  of key importance for us, so we state it explicitly:

\begin{lemma} \label{lem-proto}
Every tiling $\Sk$ in the space $X_\Tk$ has the same collection of prototiles as $\Tk$.
\end{lemma}

We also have the following

\begin{lemma} \label{top-conj} Let $\Tk$ be a tiling and $\Lb (= \Lb_\Tk)$ the corresponding Delone $\kappa$-set. Then the associated dynamical systems $(X_\Tk,\R^d)$ and $(X_{\Lbs},\R^d)$ are topologically conjugate.
\end{lemma}

The proof is immediate, since the map $\Sk \mapsto \Lb_{\Sk}$ (with a fixed set of prototiles as subsets of $\R^d$) provides the required conjugacy.




\section{Unique ergodicity} \label{unique-ergodicity}

In the case of Delone $\kappa$-set with FLC, it is shown in \cite{LMS1} that unique cluster frequencies (UCF) is an equivalent property for the dynamical system to be uniquely ergodic. However, without assuming FLC, the relation between the two is less clear. A simple observation shows that UCF does not imply unique ergodicity, without extra assumptions. Indeed, consider a ``random'' Delone $\kappa$-set $\Lb$ (say, on $\R$). With probability one, every patch will occur only once (up to translation) in $X$, hence it has uniform frequency equal to zero, but there is no reason why the associated dynamical system should be uniquely ergodic.

So, in addition to UCF, we impose another, technical condition, and show that together they 
imply the unique ergodicity of the dynamical system corresponding to the Delone $\kappa$-set.    

\medskip

{Let us define cylinder sets in $X_{\Lbs}$. In order to do this, we first introduce the sets $\Delta_m$, $\Delta_{m, \alpha}$, and $U_{m, \alpha}$ below.}
Let $\eta({\Lb}) > 0$ be chosen so that every ball of radius
$\frac{\eta({\Lbs})}{2}$
contains at most one point of $\supp({\Lb})$. 
For any given $m \in \N$, we consider a cube $[-2^m, 2^m)^d \subset \R^d$ and make a grid subdividing it into cubes of side length $\frac{1}{2^m}$. The small cubes are products of intervals which are half-open; every such cube can be written as $\prod_{i=1}^d [\frac{k_i}{2^m}, \frac{k_i+1}{2^m})$ for some $k_i \in [-4^m, 4^m)$.
Choose $m_0 \in \Z_{+}$ for which $\frac{1}{2^{m_0}} < \frac{\eta({\Lbs})}{2}$ so that each cube of the subdivision of side length $2^{-m}$ contains at most one point of $\supp({\Lb})$. We will assume that $m \ge m_0$. 

Let $\Delta_m$ be the set of all small cubes from this grid, with colours assigned from $1$ to $\kappa$. Formally, an element of $\Delta_m$ is $C=(B,i)$, where $B$ is a cube and $i\in \{1,\ldots,\kappa\}$. We write
$B = \supp(C)$. Observe that the cardinality of $\Delta_m$ is $|\Delta_m| = \kappa\cdot 2^{(2m+1)d}$. Next we consider the collection of subsets of $\Delta_m$, characterized by the property that each cube may be chosen at most once.
We denote these subsets by $\Delta_{m, \alpha}$, with $\alpha =1,\ldots,N_m$. Formally, their  property is that
\[\forall \ C_1, C_2 \in \Delta_{m, \alpha}, \ \supp(C_1) \cap \supp(C_2) = \emptyset. \]
For each $\Delta_{m,\alpha}$ let 
\begin{eqnarray*} 
U_{m, \alpha} = \bigl\{ \Qb & :  & \ \Qb \subset [-2^m, 2^m)^d \times \{1, \dots, \kappa \}, \  \mbox{finite set of coloured points, such that},  \\
& & \forall \ C \in \Delta_{m, \alpha} \ \mbox{we have}\ \Qb \cap C  = \{(p,i)\}, \  \mbox{for some coloured point $(p,i)$ in $\R^d$};
 \\
& & \forall \ C' \in \Delta_m \backslash \Delta_{m, \alpha} \ \mbox{we have}\ \Qb \cap C'  = \emptyset \bigr\}.
\end{eqnarray*}
In words, $U_{m,\alpha}$ is the family of $\kappa$-sets $\Qb$ of coloured points of cardinality $|\Delta_{m,\alpha}|$ such that  each coloured cube of $\Delta_{m,\alpha}$ contains exactly one point of $\Qb$ of its colour.
Let us consider all the clusters in $X_{\Lbs}$ which arise from $\Gb \cap [-2^m, 2^m)^d$ and belong to $U_{m, \alpha}$, for some $\Gb \in X_{\Lbs}$, that is,
\begin{eqnarray} \label{def-Xma}
X(U_{m, \alpha}) &:= & \{ \Gb \in X_{\Lbs}: \Gb \cap [-2^m, 2^m)^d \in U_{m, \alpha}  \} . 
\end{eqnarray}
{This is a cylinder set in $X_{\Lbs}$ that we define.}
Clearly, for any $m \in \N$, 
\bee \label{decomp1}
 X_{\Lbs} = \bigsqcup_{\alpha = 1}^{N_m} X(U_{m, \alpha}),
\eee
where $\bigsqcup$ denotes disjoint union.

Let $\Lb$ be a Delone $\kappa$-set. It has countably many clusters (finite subsets). For every $(m,\alpha$) we define 
\begin{eqnarray*}
\mathcal{G}_{m, \alpha} & = & \{ \Gbb \in U_{m,\alpha} : \exists \ x \in \R^d \ \mbox{s.t.} \ 
\Gbb - x \in \Lb   \}. 
\end{eqnarray*}
We identify clusters in $\mathcal{G}_{m, \alpha}$ that are translationally equivalent and choose one representative from each equivalence class;  denote the resulting set of clusters by $(\mathcal{G}_{m, \alpha})_{_{\equiv}}$.
 Enumerate elements of this set arbitrarily and write
 $$(\mathcal{G}_{m, \alpha})_{_{\equiv}} = \{\Gbb^{(m,\alpha)}_1, \Gbb^{(m,\alpha)}_2, \Gbb^{(m,\alpha)}_3, \dots \}.
 $$
 For each cluster $\Gbb$ of $\Lb$ and $V \subset \R^d$, consider the cylinder set defined in the usual way:
\[ X(\Gbb, V) = \{ \Gb \in X_{\Lbs} : \exists \ t \in V \ \mbox{s.t.}  -t+ \Gbb  \subset \Gb \}. \]
Recall that for $V \subset \R^d$, we use the notation
\[ \Gbb \textcircled{+} V = \{ \Gbb + t  \ : \ t \in V \}. \]
For every $j\ge 1$ and $\Gbb^{(m,\alpha)}_j \in (\mathcal{G}_{m, \alpha})_{_{\equiv}}$, there exists a Borel set $V^{(m,\alpha)}_j  \subset \R^d$ such that 
\begin{eqnarray} \label{choice-of-x1}
\Gbb_j^{(m,\alpha)} \textcircled{+} V_j^{(m,\alpha)}& \subset & U_{m, \alpha} \ \ \ \mbox{but}  \nonumber\\
 \Gbb_j^{(m,\alpha)} \textcircled{+} {V}' & \not\subset & U_{m, \alpha} \ \ \ \mbox{for any ${V}' \supsetneq V_j^{(m,\alpha)}$}. 
\end{eqnarray}
Indeed, by definition,  $\Gbb^{(m,\alpha)}_j \in U_{m,\alpha}$. Then $V^{(m,\alpha)}_j$ is the maximal set of all possible ``wiggle vectors $x$" such that all the points of $-x+\Gbb_j^{(m,\alpha)}$ remain inside their small grid boxes. 
This way, we get for all $(m,\alpha)$,
\bee \label{decomp2}
 X(U_{m, \alpha}) = \bigsqcup_{j=1}^{\infty} X(\Gbb^{(m,\alpha)}_j, V^{(m,\alpha)}_j) \ \bigsqcup \wtil{X}_{m,\alpha},
\eee
where
\bee
 \wtil{X}_{m,\alpha}:= \bigl\{\Gb \in X_{\Lbs} : \Gb \cap [-2^m, 2^m)^d \in U_{m, \alpha}, \nexists \ x \in \R^d \ \mbox{s.t.} \ \bigl(\Gb \cap [-2^m, 2^m)^d \bigr) + x \subset \Lb \bigr\}
\eee
is the set of coloured Delone $\kappa$-sets whose clusters in $[-2^m,2^m)^d$ are only ``admitted in the limit'',
according to the terminology of \cite{FraSa2}. Observe that, by construction, choosing $z=(z_i)_{i=1}^d$ in $\Gbb_j^{(m,\alpha)}\cap \prod_{i=1}^d [\frac{k_i}{2^m}, \frac{k_i+1}{2^m})$ for some $k_i \in [-4^m, 4^m) \cap \Z$ (ignoring the colour), 
we have 
$$
V_j^{(m,\alpha)}\subset \prod_{i=1}^d \bigl[-z_i + k_i \cdot 2^{-m}, -z_i + k_i \cdot 2^{-m} + 2^{-m}\bigr),
$$
hence
\bee \label{Vbound}
\Vol(V_{j}^{(m,\alpha)}) \le 2^{-md}\ \ \mbox{for all}\ m,\alpha,j.
\eee


The following is a standard result in topological dynamics; see e.g.\ \cite[6.19]{Walters} for the proof in the case of $\Z$-actions; the case of $\R^d$-actions is proved the same way --- see e.g.\ \cite[Thm.\,2.6]{LMS1} for a short proof of the implication needed below.


\begin{theorem} \label{th-UE}
Let $(X,T_g)_{g\in \R^d}$ be a continuous $\R^d$-action on a compact metric space and let $\{F_n\}_{n\ge 1}$ be a van Hove sequence in $\R^d$. This system is uniquely ergodic if and only if for all continuous functions
$f:X\to \C$ (notation $f\in \Ck(X)$),
\bee \label{unconv}
I_n(x,f): = \frac{1}{\mbox{\em Vol}(F_n)} \int_{F_n} f(T_g(x))\, dg \ \longrightarrow \ C_f, \ \mbox{as} \ n \to \infty,
\eee
where $C_f$ is a constant depending on $f$, for all $x\in X$, in which case the convergence is necessarily uniform in $x\in X$.
\end{theorem}

Now we are ready to state our first result.

\begin{theorem} \label{th-UE2}
Let $\Lb$ be a Delone $\kappa$-set and let $\{F_n\}_{n \ge 1}$ be a van Hove sequence. 
Suppose that $\Lb$ has UCF, and in addition,  for all $m\ge m_0$ and all $\alpha = 1,\ldots, N_m$ we have
 \bee \label{lemeq}
 \lim_{n\to \infty} \sum_{j=1}^{\infty} \frac{\Vol(V^{(m,\alpha)}_j) L_{\Gbbs_j^{(m,\alpha)} }(h + F_n) }{\Vol(F_n)} = \sum_{j=1}^{\infty} \Vol(V^{(m,\alpha)}_j) \cdot \freq(\Gbb_j^{(m,\alpha)}, \Lb),
 \eee
 uniformly in $h\in \R^d$.  
Then $(X_{\Lbs}, \R^d)$ is uniquely ergodic.
\end{theorem}

\begin{proof} We are going to use the sufficient condition from Theorem~\ref{th-UE}. Since $X_{\Lbs}$ is the orbit closure of $\Lb$ and continuous functions on the compact space $X_{\Lbs}$ are uniformly continuous, it is enough to verify (\ref{unconv}) for $x = -h+\Lb$, uniformly in $h\in \R^d$. Next observe that,  in view of the decomposition (\ref{decomp1})  of $X_{\Lbs}$ for each $m\ge m_0$ into a finite disjoint union of Borel subsets whose diameters tend to zero as $m\to \infty$, we can approximate any $f\in \Ck(\Lb)$ by linear combinations of characteristic functions of $X(U_{m,\alpha})$. Thus it is enough to show that for all $(m,\alpha)$,
\bee \label{unconv2}
I_n(-h+\Lb,\Chi_{m,\alpha}) \to  C_{m,\alpha}, \ \mbox{as} \ n \to \infty,
\eee
uniformly in $h\in \R^d$, where $\Chi_{m,\alpha}$ is the characteristic function of $X(U_{m,\alpha})$. 

Fix $m\ge m_0$ and $\alpha\in \{1,\ldots,N_m\}$. To simplify the notation,
we will drop the superscript $(m,\alpha)$ for the rest of the proof, writing $\Gbb_j:= \Gbb^{(m,\alpha)}_j,\ V_j:= V^{(m,\alpha)}_j$, etc.

Now,  observe that for any $h \in \R^d$,
\begin{eqnarray*}
J_n(h, \Chi_{m,\alpha}) & := & \int_{F_n} \Chi_{m,\alpha}(-x-h+\Lb)\, dx \\
& = & \Vol \bigl\{ x \in F_n : -x -h + \Lb \in X(U_{m, \alpha})\bigr\} \\
& = & \Vol \bigl\{ x \in h + F_n : -x + \Lb \in X(U_{m, \alpha})\bigr\} \\
& = & \Vol \Bigl\{\bigsqcup_{j=1}^{\infty} \bigsqcup_{\nu} \bigl((h+F_n) \cap (x_{j,\nu} - V_j)\bigr)\Bigr\},
\end{eqnarray*} 
where $x_{j,\nu}$ are all the vectors such that $x_{j,\nu} + \Gbb_j \subset \Lb$. 
In fact, 
\begin{eqnarray*}
\lefteqn{-x + \Lb \in X(U_{m, \alpha}) } \\ 
& \Longleftrightarrow & -x + \Lb \in X(\Gbb_j, V_j) = \{\Gb \in X_{\Lbs} : \exists \ t \in V_j \ \mbox{s.t.} \ t+\Gbb_j  \subset \Gb  \} \\
& & \ \ \ \ \ \ \ \ \ \ \ \  \mbox{for some $ X(\Gbb_j, V_j)  \subset X(U_{m, \alpha})$} \\
& \Longleftrightarrow  &  \exists \ t \in V_j \ \mbox{s.t.} \ t+ \Gbb_j  \subset -x + \Lb \\
& \Longleftrightarrow  &  \exists \ t \in V_j \ \mbox{s.t.} \ (x + t) + \Gbb_j \subset  \Lb \\
& \Longleftrightarrow  &  \exists \ t \in V_j \ \mbox{s.t.} \ x + t = x_{j,\nu}, \ \mbox{where} \ x_{j,\nu} +  \Gbb_j \subset  \Lb.
\end{eqnarray*}  
Observe that the set $\wtil{X}_{m,\alpha}$ from (\ref{decomp2}), which has patches admissible in the limit, does not enter into consideration, since we are integrating over the orbit of $\Lb$.

Recall that every $m$-grid box contains at most one point of $\supp(\Lb)$, hence for $-x+\Lb\in X(U_{m,\alpha})$ there is a unique $j\in \N$ and a unique  $t\in V_j$ such that
$(x-t)+\Gbb_j \subset \Lb$. Thus
the sets $x_{j,\nu} + V_j$ are disjoint for $\alpha = 1,\ldots, N_m$ and $j\ge 1$.
Note that $2^{m+1} \ge \max\{\diam(\Gbb_j):\, j \in \N \}$ (we are using the $\ell^\infty$ metric in $\R^d$).  Let $r = \eta(\Lb) + 2^{m+1}$. Then we obtain
$$
\sum_{j=1}^{\infty} \Vol(V_j) L_{\Gbbs_j} (h + {(F_n)}^{-r})
 \le J_n(h, \Chi_{m,\alpha}) \le \sum_{j=1}^{\infty} \Vol(V_j) L_{\Gbbs_j} (h + {(F_n)}^{+r}) \,, 
 $$
 where $F^{\pm}$ are defined in (\ref{nota1}).
In view of (\ref{Vbound}), we deduce
 \begin{eqnarray} & &  \sum_{j=1}^{\infty} \Vol(V_j) (L_{\Gbbs_j} (h + {(F_n)}^{+r}) - L_{\Gbbs_j} (h + {(F_n)}^{-r})) \nonumber \\
&  \le &   \ \ \sum_{j=1}^{\infty} 2^{-md} \cdot L_{\Gbbs_j} (h + (\partial{F_n})^{+2r}) \label{claim2} \,. 
\end{eqnarray}
Notice that  the infinite sums above have finitely many non-zero terms, since there are finitely many different clusters in a finite domain, hence there is no problem with convergence.
 
It is clear that for any $(m,\alpha)$, the total number of $\Lb$-clusters from $U_{m,\alpha}$ contained in a set $A$ is not greater than the cardinality of $\supp(\Lb)\cap A$.
 It follows that for any bounded Borel set $A$ we have
 $$
 \sum_{j=1}^\infty L_{\Gbbs_j}(A) \le \frac{\Vol(A^{+2^{-m}})}{2^{-md}}.
 $$
 Thus
 \bee \label{claim22}  \sum_{j=1}^{\infty} L_{\Gbbs_j} (h + (\partial F_n)^{+2r}) \le  2^{md}\cdot \Vol((\partial F_n)^{+(2r+2^{-m})}) \,.\eee
Since $\{F_n\}_{n \ge 1}$ is a van Hove sequence, we conclude that
 \bee \label{claim3}
 \lim_{n \to \infty} \left(\frac{J_n(h, \Chi_{m,\alpha})}{\Vol{(F_n)}}- \sum_{j=1}^{\infty} \frac{\Vol(V_j) L_{\Gbbs_j} (h + F_n) }{\Vol(F_n)}\right) = 0,
 \eee
 uniformly in $h\in \R^d$.
  Denote the right-hand side of (\ref{lemeq}) by $C_{m,\alpha}$. 
The equation  (\ref{lemeq}), together with (\ref{claim3}), shows the desired convergence of $I_n(-h +\Lb,\Chi_{m,\alpha}) = J_n(h, \Chi_{m,\alpha})/\Vol(F_n)$, uniform in $h\in \R^d$, as $n\to \infty$.  \end{proof}

The following sufficient condition for unique ergodicity will be useful in the next section.

\begin{cor} \label{cor-tech}
Let $\Lb$ be a Delone $\kappa$-set and let $\{F_n\}_{n \ge 1}$ be a van Hove sequence. 
Suppose that $\Lb$ has UCF, and in addition,  for all $m\ge m_0$ and all $\alpha = 1,\ldots, N_m$ we have
\be \label{tech_cond}
\forall\,\eps>0,\exists\, k_0\in \N, \forall\, n\ge k_0,\forall\,h\in \R^d, \ \sum_{j=k_0}^\infty \frac{\Vol(V^{(m,\alpha)}_j) L_{\Gbbs_j^{(m,\alpha)} }(h + F_n) }{\Vol(F_n)} \le \eps.
\ee
Then $(X_{\Lbs},\R^d)$ is uniquely ergodic.
\end{cor}

\begin{proof}
 We only need to check the equality (\ref{lemeq}). Observe that if we replace the infinite sum by any partial sum in (\ref{lemeq}), convergence claimed holds  by the UCF property. Therefore,
 \be \label{tech2}
\liminf_{n\to \infty,\ h\in \R^d} \sum_{j=1}^{\infty} \frac{\Vol(V^{(m,\alpha)}_j) L_{\Gbbs_j^{(m,\alpha)} }(h + F_n) }{\Vol(F_n)} \ge \sum_{j=1}^{\infty} \Vol(V^{(m,\alpha)}_j) \cdot \freq(\Gbb_j^{(m,\alpha)}, \Lb),
\ee
where $h$ is allowed to vary with $n$.
Thus it remains to estimate the $\limsup$ from above. Given $\eps>0$, find the  corresponding $k_0$ from (\ref{tech_cond}). 
 There exists $n_0$ such that for all $n\ge n_0$ and all $h\in \R^d$,
 $$
 \left|\sum_{j=1}^{k_0-1} \frac{\Vol(V^{(m,\alpha)}_j) L_{\Gbbs_j^{(m,\alpha)} }(h + F_n) }{\Vol(F_n)} - \sum_{j=1}^{k_0-1} \Vol(V^{(m,\alpha)}_j) \cdot \freq(\Gbb_j^{(m,\alpha)}, \Lb)\right| < \eps.
 $$
 Then for $n\ge \max\{n_0,k_0\}$, for all $h\in \R^d$,  we  have by the above, in view of (\ref{tech_cond}),
$$
 \sum_{j=1}^{\infty} \frac{\Vol(V^{(m,\alpha)}_j) L_{\Gbbs_j^{(m,\alpha)} }(h + F_n) }{\Vol(F_n)}  \le  \sum_{j=1}^{\infty} \Vol(V^{(m,\alpha)}_j) \cdot \freq(\Gbb_j^{(m,\alpha)}, \Lb) + 2\eps,
$$
and  the desired claim follows.
 \end{proof}


\section{Substitution Delone $\kappa$-sets and tilings} \label{substitution-tilings}

\subsection{Preliminaries}

We say that a linear map $Q:\ \R^d \to \R^d$ is {\em expansive} if there is a $c > 1$ with 
\[ \varrho(Qx, Qy) \ge c \cdot \varrho(x, y) \]
for all $x, y \in \R^d$ and some metric $\varrho$ on $\R^d$ equivalent to the Euclidean metric. Another way to define an expansive linear map is to say that all its eigenvalues (real and complex) have modulus strictly greater than 1.

\begin{defi} \label{def-subst-mul}
{\em $\Lb = (\Lam_i)_{i\le \kappa}$ is called a {\em
substitution Delone $\kappa$-set} if $\Lb$ is a Delone $\kappa$-set and
there exist an expansive map
$Q:\, \R^d\to \R^d$ and finite sets $\Dk_{ij}$ for $i,j\le \kappa $ such that
\be \label{eq-sub}
\Lambda_i = \bigcup_{j=1}^{\kappa} (Q \Lambda_j + \Dk_{ij}),\ \ \ i \le \kappa,
\ee
where the unions on the right-hand side are disjoint. 
}
\end{defi}

\begin{defi}\label{def-subst}
{\em Let $\Ak = \{T_1,\ldots,T_{\kappa} \}$ be a finite set of prototiles in $\R^d$, with $T_i=(A_i,i)$.
Denote by $\Pk_{\Ak}$ the set of
patches whose every tile is a translate of one of $T_i$'s.
We say that $\omega: \Ak \to \Pk_{\Ak}$ is a {\em tile-substitution} (or simply
{\em substitution}) with
expansive map $Q$ if there exist finite sets $\Dk_{ij}\subset \R^d$ for
$i,j \le \kappa$, such that
\begin{equation}
\om(T_j)=
\{u+T_i:\ u\in \Dk_{ij},\ i=1,\ldots,\kappa \}
\label{subdiv}
\end{equation}
with
\be \label{eq-til} 
Q A_j = \bigcup_{i=1}^{\kappa} (\Dk_{ij}+A_i) \ \ \  \mbox{for} \  j \le \kappa.
\ee
Here all sets in the right-hand side must have disjoint interiors;
it is possible for some of the $\Dk_{ij}$ to be empty.  
The {\em substitution $\kappa\times\kappa$ matrix} $\Sf$ is defined by $\Sf(i,j) = \# \Dk_{ij}$. } 
\end{defi}

The tile-substitution is extended to translated prototiles by 
\be \label{transl}
\om(T_j - x) = \om(T_j) - Qx.
\ee 
The equations (\ref{subdiv})-(\ref{transl}) allow us to extend $\om$ to patches in $\Pk_\Ak$ by
$\om(P) = \bigcup_{T\in P} \om(T)$. It is similarly extended to tilings all of whose tiles are translates of the prototiles from $\Ak$. A tiling $\Tk$ satisfying $\om(\Tk) = \Tk$ is called a
{\em fixed point of the tile-substitution}, or a {\em substitution tiling with expansion map $Q$}. 

We avoid the term ``self-affine tiling,'' since these are usually assumed to have FLC.
{A substitution tiling $\Tk$ is said to be {\em primitive}, if there is an $\ell > 0$ for which $\Sf^{\ell}$ has no zero entries, where $\Sf$ is the substitution matrix.

It is easy to see that one can always find a periodic point for $\om$, i.e.,\ a tiling $\Tk$ such that $\om^N(\Tk) = \Tk$ for some $N \ge 1$.
Indeed, since the prototile set is finite and the map $Q$ is expansive, we can find $T_j\in \Ak$ and $N\ge 1$ such that $\om^N(T_j)$ contains a translated copy of $T_j$ in the interior of its support.
This means that there exists a vector $x\in \Int(\supp(T_j))$ such that $T_j-x\in \om^N(T_j-x)$. Then the sequence of patches $\{\om^{Nk}(T_j-x)\}_{k\ge 1}$ is increasing, in the sense that each one contains the previous one as a sub-patch, and covers the entire space in the limit. This limit is the desired periodic point $\Tk$.
We can then replace $\om$ with $\om^N$ and obtain a tiling $\Tk$ that is a fixed point of a tile-substitution.

\medskip

Comparing the systems of equations (\ref{eq-sub}) and (\ref{eq-til}), we can see that there is a kind of ``duality'' between them. Following \cite{lawa}, we call the system (\ref{eq-til}) the {\em adjoint system} of equations for (\ref{eq-sub}), where we use the same expansion map $Q$ and the same  sets of translation vectors $\Dk_{ij}$. By the theory of graph-directed iterated function systems \cite{MauWi} (see also \cite{BM1}), it follows that the system of equations (\ref{eq-til}) 
has a unique solution for which $\{A_1, \dots, A_{\kappa} \}$ is a
family of non-empty compact subsets of $\R^d$.

The following result  of J. Lagarias and Y. Wang will be important for us; it is a special case of \cite[Th.\,2.3 and Th.\,5.5]{lawa}.
We emphasize that in their paper the FLC condition is not assumed.

\begin{theorem}[Lagarias and Wang \cite{lawa}] \label{th-lawa}
Suppose that $\{A_1, \dots, A_{\kappa} \}$ is the solution of the adjoint system of equations (\ref{eq-til}) associated with a primitive substitution Delone $\kappa$-set. Let $Q$ be the expansion map of the substitution and $\Sf$ the substitution matrix. Then

{\rm (i)} the Perron-Frobenius eigenvalue of the matrix $\Sf$ equals $|\det(Q)|$;

{\rm (ii)} all the sets $A_i$ have non-empty interiors and each $A_i$ is the closure
of its interior;

{\rm (iii)} for all $1 \le i \le \kappa$, $\Vol(\partial A_i) = 0$.
\end{theorem}

Note that for substitution tilings (without the FLC assumption) the property (iii) was shown in \cite{Prag}.

\smallskip

By Theorem~\ref{th-lawa}, given a substitution Delone $\kappa$-set $\Lb$, we have a natural candidate for a tiling, namely, $\Lb + \Ak := \{x + T_i :\ x\in \Lambda_i,\ i \le \kappa \}$, where $T_i = (A_i, i)$, $i\le \kappa$. It is shown in \cite[Ex.\ 3.12]{lawa} that, in general, this may not be a tiling, but only a covering. 

\begin{defi}
A substitution Delone $\kappa$-set $\Lb$, satisfying (\ref{eq-sub}) is said to be {\em representable} (by tiles) if $\Lb + \Ak $
is a tiling of $\R^d$, where $T_i = (A_i,i)$, $i \le \kappa$, and
$A_i$'s arise from the solution to the adjoint system
(\ref{eq-til}) and $\Ak = \{T_i : i\le \kappa \} $. If $\Lb$ is representable, then $\Tk:=\Lb + \Ak$ is a substitution tiling with expansion $Q$ and
we call it the {\em associated substitution tiling} of
$\Lb$. 
\end{defi}

Observe that, assuming $\Lb$ is representable, there is a one-to-one correspondence between $\Lb$-clusters and $\Tk$-patches: if $\Gbb= (G_i)_{i\le \kappa}$ is a $\Lb$-cluster, then $P = \Gbb + \Ak := \{x + T_i : x \in G_i, i \le \kappa \} $ is the associated $\Tk$-patch, and this procedure can be reversed.

A criterion for $\Lb$ to be representable 
was found in \cite{LMS2}, extending sufficient conditions from \cite{lawa}. We recall it here for reader's convenience, referring to \cite{lawa,LMS2} for details.

\begin{defi}{
For a substitution Delone $\kappa$-set $\Lb = (\Lambda_i)_{i \le \kappa}$ satisfying (\ref{eq-sub}), define a matrix $\Phi=(\Phi_{ij})_{i,j=1}^{\kappa}$ 
whose entries are finite (possibly empty) families of linear affine transformations on 
$\R^d$ given by
$$
\Phi_{ij} = \{f: x\mapsto Qx+a:\ a\in \Dk_{ij}\}\,.
$$
We define $\Phi_{ij}(\Xk) := \bigcup_{f\in \Phi_{ij}} f(\Xk)$ for  $\Xk\subset \R^d$. For a $\kappa$-set $(\Xk_i)_{i\le \kappa}$ let
$$
\Phi\bigl((\Xk_i)_{i\le \kappa}\bigr) = \Bigl(\bigcup_{j=1}^{\kappa} \Phi_{ij}(\Xk_j)\Bigr)_{i\le \kappa}\,.
$$
Thus $\Phi(\Lb)= \Lb$ by definition. We say that $\Phi$ is a {\em $\kappa$-set substitution}.
}
\end{defi}

Let $\Lb$ be a substitution Delone $\kappa$-set and $\Phi$ the associated $\kappa$-set substitution.

\begin{defi}
{\em  Let $\Lb$ be a primitive substitution Delone $m$-set and let
$\Gbb$ be a cluster of $\Lb$. The
cluster $\Gbb$ will be called {\em legal} if it is a translate of a sub-cluster of
$\Phi^k(\{x_j\})$ for some $x_j \in \Lam_j$, $j \le \kappa$ and $k \in \N$. (Here $\{x_j\}$ is a $\kappa$-set which is empty in all coordinates other than $j$, for which it is a singleton.)}
\end{defi}

\begin{theorem}\cite{LMS2}\label{legal-rep}
Let $\Lb$ be a repetitive primitive substitution Delone $\kappa$-set. $\Lb$ is representable if and only if  every $\Lb$-cluster is legal.
\end{theorem}

\begin{remark} {\em
In \cite[Lemma 3.2]{lawa} it is shown that if $\Lb$ is a substitution Delone $\kappa$-set,
then there is a finite $\kappa$-set (cluster) $\Gbb \subset \Lb$ for which
$\Phi^{n-1}(\Gbb) \subset \Phi^n(\Gbb)$ for $n \ge 1$ and 
$\Lb = \lim_{n \to \infty} \Phi^n (\Gbb)$. We call such a $\kappa$-set $\Gbb$ 
a {\em generating $\kappa$-set}.
Note that, in order to check that every $\Lb$-cluster is legal, 
we only need to see if some cluster that contains a finite generating 
$\kappa$-set for $\Lb$ is legal.}  
\end{remark}

\subsection{Supertiles and recognizability} For much of this section, we mostly work with substitution tilings, since the notions of supertiling and tile frequencies are more natural in the tiling setting. We then obtain the results for representable substitution Delone $\kappa$-sets as an immediate consequence.

\smallskip

Let $\Tk$ be a fixed point of a primitive tile-substitution $\om$.

\begin{lemma} \label{lem-subs}
The tile-substitution $\om$ extends to a continuous map $\om:\,X_\Tk\to X_\Tk$, which is a surjection.
\end{lemma}

\begin{proof}
It follows from definitions that the tile-substitution $\om$ is
well-defined on the set of tilings, all of whose tiles are translates of the prototiles from $\Ak$, and $\om$ is continuous on this set. By Lemma~\ref{lem-proto}, the mapping $\om$ is defined on the tiling space $X_\Tk$.
 It follows from (\ref{transl}) that
\bee\label{intertwine}
\om(\Sk-x) = \om(\Sk)-Qx\ \ \mbox{for all}\ \Sk\in X_\Tk,\ x\in \R^d.
\eee
By  definition, we have $\om(\Tk)=\Tk$, and by (\ref{intertwine}), the orbit of $\Tk$ is mapped onto the orbit, hence by continuity we
have $\om(X_\Tk) = X_\Tk$, as desired.
\end{proof}


It follows from surjectivity of $\om$ that
we can compose the tiles of $\Sk\in X_\Tk$ 
into {\em supertiles of order $k$}, for any $k\ge 1$, although not necessarily uniquely. Formally, for $\Sk\in X_\Tk$ and $k\in\N$ there exists $\Sk^{(k)}$ such that $\om^k(\Sk^{(k)}) = \Sk$.
Then $\{Q^k T:\ T\in \Sk^{(k)}\}$
are the supertiles 
of $\Sk$ of order $k$, and the convention is that $Q^k T$ inherits the type from $T$.

Denote by $\Kk(\Tk)$ the group (a subgroup of $\R^d$) of translational periods of $\Tk$:
$$
\Kk(\Tk) = \{x\in \R^d:\ \Tk = \Tk-x\}.
$$
By the definition of $X_\Tk$ as the orbit closure of $\Tk$ we have that $\Kk(\Sk) \supset \Kk(\Tk)$ for all $\Sk\in X_\Tk$.
It follows from (\ref{intertwine}) that
$$
\om(\Sk) = \om(\Sk-x),\ \ \Sk\in X_\Tk,\ x\in Q^{-1}\Kk(\Tk).
$$
Since $Q$ is expanding and $\Kk(\Tk)$ is discrete, we see that, in the case when the group of translational periods is non-trivial, that is,
$\Kk(\Tk) \ne \{0\}$, the tile-substitution $\om$ 
is not 1-to-1 on $X_\Tk$.


\begin{defi} \label{def-recog}
{\em The tile-substitution $\om$ is said to be {\em recognizable} if $\om$ is one-to-one, in which case $\om$ is a homeomorphism of $X_\Tk$.
The tile-substitution $\om$ is called {\em recognizable modulo periods} if $\om(\Sk) = \om(\Sk')$ implies $\Sk'=\Sk-Q^{-1}x$ for some $x\in \Kk(\Tk)$.
}
\end{defi}

 Recognizability in the non-periodic primitive substitution FLC case and recognizability modulo periods in the general primitive substitution FLC case was proved in 
\cite{sol-ucp}, as a generalization of \cite{mosse} from the 1-dimensional symbolic substitution setting. Recognizability in specific non-FLC examples is usually easy to verify by inspection, but we do not know of a general criterion.

\subsection{Primitivity, repetitivity, and minimality}
We discuss these issues briefly, since there is often some confusion around them, and there are additional subtleties in the ILC case. 
In the FLC setting this topic is treated thoroughly in \cite[Section 6]{Robi.lec}. In the non-FLC case, in a more general setting than ours, both in terms of spaces of Delone sets and in terms of groups acting on them, these questions are studied in [Fre-Ric].

Let $\Tk$ be a fixed point of a tile-substitution.
Recall that a dynamical system $(X_\Tk,\R^d)$ is minimal if its every orbit is dense. Minimality implies that the substitution is primitive. Indeed, otherwise there is a prototile $T_i$, a type $j\in \{1,\ldots,\kappa\}$, and a sequence $n_k\to \infty$ such the patches
$\om^{n_k}(T_i)$ do not contain any tile of type $j$. We can then find $x_k\in \R^d$ such that $\Tk-x_k$ has a translate of the patch $\om^{n_k}(T_i)$ covering a ball $B_{R_k}(0)$, with $R_k\to \infty$. Then a subsequential limit of $\Tk-x_k$ will
be a tiling  $\Sk\in X_\Tk$ without tiles of type $j$, so clearly the orbit of $\Sk$ is not dense.

Recall that a tiling $\Tk$ is said to be repetitive if every $\Tk$-patch appears relatively dense in space. If $\Tk$ is a repetitive fixed point of a tile-substitution, then it is immediate that the tile-substitution is primitive.
The converse does not hold, even for FLC tilings, see \cite[Example 5.11]{Robi.lec}.

It is also not hard to see that if $\Tk$ is repetitive, then $X_\Tk$ is minimal (see \cite[Prop.\,3.1]{FraSa2} for the proof of the implication ``primitivity + repetitivity $\Rightarrow$ minimality'' in the more general setting of fusion tilings).
So far, everything is exactly as in the case of FLC tilings. However, whereas for FLC tilings minimality of $X_\Tk$ is {\em equivalent} to repetitivity, 
for non-FLC tilings, minimality of $X_\Tk$ is {\em equivalent} to `almost repetitivity' introduced in \cite{Fre-Ric}. The almost repetitivity may be  defined as follows: a tiling $\Tk$ is almost repetitive if and only if the set
$\Vk_\eps(\Tk):=\{x\in \R^d:\ \varrho(\Tk, \Tk+x)< \eps\}$ is relatively dense in $\R^d$ for any $\eps>0$, where $\varrho$ is the local rubber metric on the tiling space.

A patch $P$ is called {\em legal} if it is a subpatch of $\om^k(T_i)$ for some $k\in \N$ and a prototile $T_i$. 
A fixed point $\Tk$ of a primitive substitution is repetitive if and only if there is a legal $\Tk$-patch containing a neighborhood of the origin in the interior of its support.
 Following \cite{FraSa2}, we say that patches that can be obtained as limits of legal patches, but are themselves non-legal, are {\em admitted in the limit}.

In \cite{FraSa2,Frank} a slightly different definition of the substitution tiling space is used:  $X_\om$, consisting of all tilings whose every patch can be obtained as a {\em limit} of a legal patch.  
We always have $X_\om\subseteq X_\Tk$, and if $\Tk$ is a repetitive fixed point of a tile-substitution, then $X_\om=X_\Tk$.

\subsection{Patch frequencies}

\begin{theorem} \label{th-UCF1}
Let $\Tk$ be a fixed point of a primitive tile-substitution and $\{F_n\}_{n\ge 1}$ a van Hove sequence. Then

{\rm (i)} for any legal patch $P$, the tiling $\Tk$ has uniform  patch frequencies $\freq(P,\Tk)>0$ relative to $\{F_n\}_{n\ge 1}$, and they are independent of the van Hove sequence;

{\rm (ii)}  for any $\Sk\in X_\Tk$ and any legal patch $P$,
$$
\freq(P,\Sk) = \freq(P,\Tk)>0.
$$
Non-legal ones, including the patches that are admitted in the limit, have zero frequency for all $\Sk\in X_\Tk$.
\end{theorem}

We note that zero frequency of patches that are admitted in the limit has been shown in \cite{FraSa2}
 in the more general context of fusion tilings (but assuming recognizability, which we do not need). 
The following elementary estimate will be useful.

\begin{lemma}[see Lemma A.4(i) in \cite{LMS2}] \label{lem-ezer}
Let $\Tk$  be a fixed point of a primitive substitution, and let $F\subset \R^d$ be an arbitrary bounded set. Then for any $\Tk$-patch $P$ and any $h\in \R^d$ we have:
$$
L_P(F,\Tk) \le V_{\min}^{-1} \cdot \Vol(F),
$$
where $V_{\min}$ is the minimal volume of $\Tk$-prototile.
\end{lemma}

\begin{proof}[Proof sketch of Theorem~\ref{th-UCF1}]
(i) This was proved under the standing FLC assumption in \cite[Appendix A.1]{LMS2},
however,  the  argument did not use the FLC. Note that for any prototile $(A_i,i)$, the sequence $\{Q^n A_i\}_{n\ge 0}$ is van Hove.
This follows from the property $\Vol(\partial A_i)=0$, see Theorem~\ref{th-lawa}(iii).  Next, it is shown that, given a legal patch $P$, the limit
\begin{equation} \label{ezer2}
c_P:= \lim_{n\to \infty} \frac{L_P(Q^n A_i,\Tk)}{\Vol(Q^n A_i)}>0
\end{equation}
exists and is independent of the tile type $i$. This is similar to the proof of uniform word frequencies for primitive symbolic substitution systems (see \cite{Queff}), relying on 
the Perron-Frobenius Theorem. In the next step, we consider the decomposition of $\R^d$ into the supertiles $Q^k T$, for $T\in \Tk$, using the fact that $\Tk = \om^k(\Tk)$. Then
$L_P(h + F_n,\Tk)$ is the sum of $L_P(Q^k A,\Tk)$, where $Q_k(A)$ ranges over the supports of those supertiles of order $k$ which lie in $h+F_n$, up to the ``boundary effects'' from $h + \partial F_n$ and the boundaries $\partial Q^k A$. Then the desired claim follows from (\ref{ezer2}), Lemma~\ref{lem-ezer}, and the van Hove property.

(ii) For a legal patch, the proof proceeds along the same lines as in part (i). Let $\Sk\in X_\Tk$. We already know the existence of $c_P>0$. Recall that by Lemma~\ref{lem-subs} there exists a decomposition of $\Sk$ into supertiles of order $k$, for any $k\in \N$ (we don't need uniqueness of the decomposition, so recognizability is not needed). In these supertiles the frequency of $P$ approaches $c_P$, and what remains are boundary effects. As for the patches that are non-legal, they necessarily have to intersect the boundary of supertiles for every order $k$, hence their number is
negligible, because the sets $Q^k A_i$ have the van Hove property. 
\end{proof}

The following is now immediate.

\begin{cor} \label{cor-freq}
Let $\Lb$ be a representable primitive substitution Delone $\kappa$-set and $\Gbb\subset \Lb$ a cluster. Then for any $\Gb\in X_{\Lbs}$ we have
$$
\freq(\Gbb,\Lb) = \freq(\Gbb,\Gb).
$$
\end{cor}

\qed


\subsection{Unique ergodicity and measure of cylinder sets}
As mentioned in the introduction, the following theorem can also be deduced from \cite{FraSa2}; we present a proof using our approach.

\begin{theorem} \label{th-UE3}
Let $\Tk$ be a fixed point of a primitive tile-substitution. Then the associated topological dynamical system $(X_\Tk,\R^d)$ is uniquely ergodic. 
\end{theorem}

\begin{proof}
We are going to verify the hypothesis of Corollary~\ref{cor-tech}, translated to the tiling setting. The tiling $\Tk$ has countably many patches. Fix $r>0$ and enumerate all the patches $P_j$ whose supports have diameter $\le r$. (In the FLC case this will be a finite set.) 
Fix also a van Hove sequence $\{F_n\}_{n\ge 1}$. 
In view of Corollary~\ref{cor-tech}, 
it suffices to show that for any $\eps>0$ there exists $k_0$ such that for all $h\in \R^d$ and $n\ge k_0$,
\be \label{cond3}
\sum_{j=k_0}^\infty \frac{L_{P_j}(h+F_n, \Tk)}{\Vol(F_n)} \le \eps.
\ee
Say that a patch $P\subset \Tk$ is {\em $k$-special} and write $\sp(P)=k$ if $P$ occurs (up to translation) as a subpatch of $\om^k(T_j)$ for some $j$ and $k\in \N$ is minimal with this property.  
Note that the hierarchical structure for $\Tk$ is fixed, even if recognizability does not hold. We write $\sp(P)=\infty$ if $\sp(P)\ge k$ for all $k\in \N$. This will happen if all occurrences of the patch $P$ in $\Tk$ cross the boundary of a supertile, for any level supertiling composed from $\Tk$.
Clearly, for any $k$ there are  finitely many  $k$-special patches.  We will next estimate the total number of patches $P_j$ of diameter $\le r$, with $\sp(P_j)>k$, that are contained in $F_n + h$.
By definition, these patches must  intersect the boundary of one of the $k$-level supertiles.  Therefore,
\begin{eqnarray}
\sum_{P_j:\ \sp(P_j)>k} L_{P_j} (h+F_n, \Tk)&  \le & {\Large \#}\Bigl\{P\subset \Tk:\ \diam(P)\le  r, \nonumber\\
& & \supp(P) \subset \!\!\!\!\!\!\bigcup_{\stackrel{\scriptstyle{(A,l)\in \Tk:}} {Q^k A \subset (h+F_n)^{+r}}}\!\!\!\!\!\! \bigl(\partial (Q^k A) \bigr)^{+r}\Bigr\}. \label{eq-new1}
\end{eqnarray}
One can see that the number of $\Tk$-patches of diameter $\le r$ contained in any given set is bounded above by $C(r)$ times the volume of that set, with a  constant $C(r)$ depending only on the tiling and on $r>0$. Note that this does not follow from Lemma~\ref{lem-ezer} in the non-FLC case. An easy (although inefficient) estimate is as follows: let $V_{\min}$ be the minimal volume of a $\Tk$-prototile. Then any patch of diameter $\le r$ contains at most $\Vol(B_r)/V_{\min}$ tiles, hence any given tile belongs to at most $2^{{\rm Vol}(B_r)/V_{\min}}$ patches of diameter $\le r$. Thus we can take
$$
C(r):= \frac{2^{{\rm Vol}(B_r)/V_{\min}}}{V_{\min}},
$$
and (\ref{eq-new1}) yields
$$
\sum_{P_j:\ \sp(P_j)>k} L_{P_j} (h+F_n, \Tk) \le C(r) \cdot\!\!\!\!\!\! \sum_{\stackrel{\scriptstyle{(A,l)\in \Tk:}} {Q^k A \subset (h+F_n)^{+r}}}\!\!\!\!\!\! \Vol\Bigl(\bigl(\partial (Q^k A) \bigr)^{+r}\Bigr).
$$
Let $\delta>0$. Recall that $\{Q^k A\}_{k\ge 1}$ is a van Hove sequence, hence there exists $k_1$ such that 
$$
\Vol \Bigl(\bigl(\partial (Q^k A) \bigr)^{+r}\Bigr) \le \delta \cdot\Vol(Q^k A),
$$
for any tile support $A$ and any $k\ge k_1$. Thus
$$
\sum_{P_j:\ \sp(P_j)>k} L_{P_j} (h+F_n, \Tk) \le C(r)\cdot \delta\cdot \Vol\bigl((h+F_n)^{+r}\bigr) = C(r) \cdot \delta \cdot \Vol\bigl((F_n)^{+r}\bigr).
$$
Since $F_n$ is van Hove as well, there exists $n_1$ such that $\Vol\bigl((F_n)^{+r}\bigr) \le (1+\delta)\Vol(F_n)$, for all $n \ge n_1$. So finally, we obtain
$$
\sum_{P_j:\ \sp(P_j)>k} \frac{L_{P_j} (h+F_n, \Tk)}{\Vol(F_n)}  \le C(r)\cdot \delta(1+\delta),
$$
and since the right-hand side can be made arbitrarily small, (\ref{cond3}) follows.
\end{proof}

Now we return to substitution Delone $\kappa$-sets.

\begin{cor}\label{cor-UCF3}
Let $\Lb$ be a representable primitive substitution Delone $\kappa$-set. Then the associated topological dynamical system $(X_{\Lbs},\R^d)$ is uniquely ergodic.
\end{cor}

Let $\mu$ be the unique translation-invariant Borel probability measure on $X_{\Lbs}$.
We will need formulae for the measure of cylinder sets $X(U_{m,\alpha})$. Although it seems natural that the constant $C_{m,\alpha}$ in (\ref{unconv2}) must be $\int_{X_{\Lbt}} \Chi_{m,\alpha}\, d \mu = \mu(X(U_{m,\alpha}))$, this is not immediately clear, since the characteristic function is not continuous. The usual approach is to approximate the characteristic function by continuous functions, but this procedure is less trivial in the ILC case.

\begin{prop} Let $\Lb$ be a representable primitive substitution Delone $\kappa$-set, and let $\mu$ be the unique invariant Borel probability measure for the associated dynamical system.

{\bf (i)} For any cluster $\Gbb\subset \Lb$ and a Borel set $V\subset \R^d$ we have
\bee \label{meas1}
\mu(X(\Gbb,V) )= \Vol(V)\cdot\freq(\Gbb,\Lb).
\eee

{\bf (ii)} For $m\ge m_0$ let
$X(U_{m, \alpha})$ be as defined in (\ref{def-Xma}) and $\Gbb_j^{(m,\alpha)}$ and $V_j^{(m,\alpha)}$ are from (\ref{decomp2}). Suppose that $\Lb$ has UCF. Then 
\bee \label{meas2} 
\mu(X(U_{m, \alpha})) = \sum_{j=1}^{\infty} \Vol(V_j^{(m,\alpha)}) \cdot \freq(\Gbb^{(m,\alpha)}_j, \Lb)\,.
\eee
Moreover,
\bee \label{meas3}
\mu(\wtil{X}_{m,\alpha})=0,
\eee
where $\wtil{X}_{m,\alpha}$ is the collection of Delone $\kappa$-sets from $X_{\Lbs}$ that are admitted in the limit, see (\ref{decomp2}).
\end{prop}

\begin{proof}
(i) Consider the characteristic function of $X(\Gbb,V)$, denoted by $\chi_{_{\scriptstyle{X}(\Gbbs,V)}}$. By the Birkhoff Ergodic Theorem, we have
$$
\mu(X(\Gbb,V)) = \lim_{n\to \infty} \frac{1}{\Vol(F_n)}\int_{F_n} \chi_{_{\scriptstyle{X}(\Gbbs,V)}}(\Gb-x)\,dx,
$$
for $\mu$-a.e.\ $\Gb\in X_{\Lbs}$. The right-hand side equals $\Vol(V)\cdot\freq(\Gbb,\Gb)$; this is shown using standard estimates based on the van Hove property, similarly to the proof of Theorem~\ref{th-UE2} above. In view of Corollary~\ref{cor-freq}, the desired claim follows.

(ii) In view of the disjoint union decomposition (\ref{decomp2}) and the already proved (\ref{meas1}),
$$
\mu(X(U_{m, \alpha})) = \sum_{j=1}^{\infty} \Vol(V_j^{(m,\alpha)}) \cdot \freq(\Gbb^{(m,\alpha)}_j, \Lb) + \mu(\wtil{X}_{m,\alpha}).
$$
Note that (\ref{decomp1}) implies $\sum_{\alpha=1}^{N_m} \mu(X(U_{m, \alpha})) = \mu(X_{\Lbs})=1$.
Again using (\ref{decomp1}), we obtain $\sum_{\alpha=1}^{N_m} \Chi_{m,\alpha} = \Chi_{_{\scriptstyle{X}_{\Lbs}}}\equiv 1$. Therefore,
$
\sum_{\alpha=1}^{N_m} J_n(h, \Chi_{m,\alpha}) = \Vol(F_n).
$
Now, summing (\ref{claim3}) over $\alpha$ and passing to the limit $n\to \infty$ (which we already know, thanks to (\ref{lemeq})), yields for all $m\ge m_0$,
$$
\sum_{\alpha =1}^{N_m} \sum_{j=1}^\infty  \Vol(V_j^{(m,\alpha)}) \cdot \freq(\Gbb^{(m,\alpha)}_j, \Lb) =1,
$$
and this yields both (\ref{meas2}) and (\ref{meas3}).
\end{proof}

\section{Ergodic-theoretic properties; eigenvalues} \label{relDenseEigenvalue-PisotFamily-Meyerset}

Let $\Lb$ be a representable primitive substitution Delone $\kappa$-set, and let $\mu$ be the unique invariant Borel probability measure for the associated dynamical system.
Further, fix $m\ge m_0$ and let
$X(U_{m, \alpha})$ be the cylinder sets from (\ref{def-Xma}). 
 For $\Lb = (\Lam_i)_{i=1}^\kappa$, let
$$
\Xi(\Lb) := \bigcup_{i=1}^\kappa (\Lam_i-\Lam_i)
$$
be the set of translation vectors between points of the $\kappa$-set of the same colour. 
We will also need the corresponding substitution tiling $\Tk=\Tk_{\Lbs}$. For a tiling $\Tk$ we define
\[\Xi(\Tk) =  \{x \in \R^d  : \ \exists \ T, T' \in \Tk,\ T'=T+x \}.\] 
Clearly,  $\Xi(\Tk_{\Lbs}) = \Xi(\Lb)$. Further, let
\[\Xi_{\rm legal}(\Tk) = \{x\in \R^d:\ \exists\ T,T'\in P\subset \Tk,\ P\ \mbox{is a legal patch},\ T'= T+x\}.\]
Of course, we have $\Xi_{\rm legal}(\Tk) =\Xi(\Tk)$ if $\Tk$ is  repetitive. We also let $\Xi_{\rm legal}(\Lb) = \Xi_{\rm legal}(\Tk_{\Lbs})$.
Let $\mu$ be the unique ergodic translation-invariant measure on $X_{\Lbs}$. 

\begin{lemma} \label{intersection-of-cylinderSet}
Let $\Lb$ be a  representable primitive substitution Delone $\kappa$-set. Let $m\ge m_0$ and $\alpha\in\{1,\ldots,N_m\}$.
For $z \in \Xi_{\rm legal}(\Lb)$, there exists $\delta = \delta(z)$  independent of $(m,\alpha)$ such that 
for all $n$ sufficiently large, $n\ge n(m,\alpha)$,
\be \label{new1-mix}
\mu(X(U_{m, \alpha}) \cap {\sf T}_{-Q^n z} X(U_{m, \alpha})) > \delta \mu(X(U_{m, \alpha}))\,.
\ee
\end{lemma}

\proof Recall the decomposition (\ref{decomp2}):
$$
 X(U_{m, \alpha}) = \bigsqcup_{j=1}^{\infty} X(\Gbb^{(m,\alpha)}_j, V^{(m,\alpha)}_j) \ \bigsqcup \wtil{X}_{m,\alpha}.
$$
We fix $(m,\alpha)$ and drop the superscripts, writing $\Gbb_j:= \Gbb^{(m,\alpha)}_j$ and $V_j:= V^{(m,\alpha)}_j$ to simplify the notation.
Notice that 
$$
X\bigl(\Gbb_j, V_j\bigr) \cap \Bigl(z+X\bigl(\Gbb_j, V_j\bigr)\Bigr) \supset  X\bigl(\Gbb_j \cup (z+\Gbb_j), V_j\bigr).
$$
Thus 
\begin{eqnarray}
 \mu(X(U_{m, \alpha}) \cap \Gamma_{-Q^n z} X(U_{m, \alpha})) & \ge & \sum_{j=1}^\infty \mu (X(\Gbb_j \cup (z+\Gbb_j), V_j)) \nonumber \\
 & = & \sum_{j=1}^{\infty} \Vol(V_j)\cdot \freq(\Gbb_j \cup (Q^n z + \Gbb_j), \Lb)\,, \label{new2-mix}
\end{eqnarray}
in view of (\ref{meas1}).

As before, it is easier to estimate frequencies for the corresponding substitution tiling $\Tk=\Tk_{\Lbs}$. The argument is similar to that of \cite{soltil}, but we show the details for completeness.
Let $P_j = \Gbb_j+\Ak$ be the $\Tk$-patch corresponding to the $\Lb$-cluster $\Gbb_j$.
Fix $T = (A_k, k)$,  a $\Tk$-tile of (any) type $k$.
We have
\[ \freq(P_j \cup (Q^n z + P_j), \Tk) = \lim_{N \to \infty} \frac{L_{P_j \cup (Q^n z + P_j)}(Q^N A_k)}{\Vol(Q^N A_k)}\,.\]
Since $z \in \Xi_{\rm legal}(\Lb)=\Xi_{\rm legal}(\Tk)$, there exist $\Tk$-tiles $T_\ell, T'_\ell$ in the same legal patch, such that $T'_\ell = T_\ell + z$. Fix another $\Tk$-tile, $T_i$ of type $i$.
By legality and primitivity, there exists $k_0 \in \N$ such that $\om^{k_0}(T_i) \supset \{T_\ell, T_\ell + z \}$.
Then for any $n \in \N$ we have $\om^{n+k_0} (T_i)  \supset \om^n(T_\ell) \cap (\om^n(T_\ell) + Q^n z)$.
Thus for any $\Tk$-patch $P$ in $\om^n(T_\ell)$ equivalent to $P_j$, the patch $P \cup ( Q^n z + P)$ is in $\om^{n+ k_0} (T_i)$.
It follows that there are at least $L_{P_j}(Q^n A_\ell)$ patches in $\om^{n+k_0} (T_i)$ which are equivalent to 
$P_j \cup (Q^n z + P_j)$; in other words,
$$
L_{P_j \cup (Q^n z + P_j)} (Q^{n+k_0} A_i) \ge L_{P_j}(Q^n A_\ell).
$$
Therefore, 
for $N > n + k_0$,
\[ L_{P_j \cup (Q^n z + P_j)} (Q^N A_k) \ge L_{\{T_\ell\}} (Q^{N-n-k_0} A_k) L_{P_j}(Q^n A_\ell)\,.\]
By the definition of the substitution matrix $\Sf$ we have
\[ L_{\{T_\ell\}} (Q^{N-n-k_0} A_k) = (\Sf^{N - n - k_0})_{\ell k}\,.\]
So for any $n \in \N$ and $N > n+k_0 $,
\begin{eqnarray*}
\ \freq(P_j \cup (Q^n z + P_j), \Tk) 
& \ge &  \lim_{N\to \infty} \frac{(\Sf^{N - n - k_0})_{\ell k} \,L_{P_j}(Q^n A_\ell)}{|\det Q|^N \,\Vol(A_k)} 
\end{eqnarray*}
It follows from the Perron-Frobenius Theorem (see \cite[Cor.\,2.4]{soltil}) that
\begin{eqnarray}  \label{eigenvector-substitution}
 \lim_{N \to \infty} \frac{(\Sf^{N - n -k_0})_{\ell  k }}{\Vol(Q^N A_k)} = r_\ell |\det Q|^{-n-k_0}\,,
\end{eqnarray}
where $(r_i)_{i \le \kappa}$ is the right Perron-Frobenius eigenvector of $\Sf$ such that $\sum_{i=1}^{\kappa} r_i \Vol(A_i) = 1$.
Thus 
\begin{eqnarray*}
\freq(P_j \cup (Q^n z + P_j), \Tk) 
& \ge  &  |\det Q|^{-n-k_0} r_\ell  \cdot L_{P_j}(Q^n A_\ell), 
\end{eqnarray*}
Since \[  \freq(P_j, \Tk) = \lim_{n \to \infty} \frac{L_{P_j}(Q^n A_\ell)}{{|\det Q|^n\Vol(A_\ell)} },\] 
we obtain
\[ \liminf_{n \to \infty} \frac{\freq(P_j \cup (Q^n z + P_j), \Tk)}{\freq(P_j, \Tk)} \ge r_\ell \Vol(A_\ell)\cdot |\det Q|^{-k_0}\,.\]

\smallskip

Let $\delta = \frac{1}{4} r_\ell\Vol(A_\ell)\cdot  |\det Q|^{-k_0}$, which is independent of $P_j$, as well as of $(m,\alpha)$.
It follows from the above that for $n =n(P_j) \in \N$ sufficiently large,
 \[ \freq(P_j \cup (Q^n z + P_j), \Tk) > 2 \delta \cdot \freq(P_j, \Tk) \,. \]
Returning to the substitution Delone $\kappa$-set $\Lb$, we obtain that for $n =n(\Gbb_j) \in \N$ sufficiently large,
 \[ \freq(\Gbb_j \cup (Q^n z + \Gbb_j), \Lb) > 2 \delta \cdot \freq(\Gbb_j, \Lb) \,. \]
We can choose $J \in \N$ such that 
\[ \sum_{j = J+1}^{\infty} \Vol(V_j) \cdot \freq(\Gbb_j, \Lb) < \delta \sum_{j =1}^{J} \Vol(V_j) \cdot \freq(\Gbb_j, \Lb). \]
For any $n \ge \max\{n(\Gbb_1), \dots, n(\Gbb_J)\}$,
\[ \sum_{j = 1}^{J} \Vol(V_j)\cdot \freq(\Gbb_j \cup (Q^n x + \Gbb_j), \Lb) > 2 \delta  \sum_{j =1}^{J} \Vol(V_j) \cdot \freq(\Gbb_j, \Lb).\]
Thus 
\[ \sum_{j = 1}^{\infty} \Vol(V_j) \cdot \freq(\Gbb_j \cup (Q^n z + \Gbb_j), \Lb) > \delta \sum_{j =1}^{\infty} \Vol(V_j) \cdot \freq(\Gbb_j, \Lb).\]
Recalling (\ref{new2-mix}) and
\[ \mu(X(U_{m, \alpha})) = \sum_{j=1}^{\infty} \Vol(V_j) \cdot \freq(\Gbb_j, \Tk) \,,
\]
we obtain (\ref{new1-mix}),
as desired.
\qed

\begin{cor} \label{cor-mix}
Let $\Lb$ be  a primitive representable substitution Delone $\kappa$-set. Then the associated uniquely ergodic measure-preserving system is not strongly mixing.
\end{cor}

\begin{proof}
Let $z\in \Xi_{\rm legal}(\Lb)$, $z\ne 0$. If the measure-preserving system were mixing, we would have
\be \label{mix}
\mu(X(U_{m, \alpha}) \cap{\sf T}_{-Q^n z} X(U_{m, \alpha})) \to \mu(X(U_{m, \alpha}))^2,\ \ \ \mbox{as}\ \ n\to \infty.
\ee
It follows from the definition of $U_{m,\alpha}$ that $\lim_{m\to \infty} N_m = \infty$. Recall that $\sum_{\alpha=1}^{N_m} \mu(X(U_{m,\alpha}))=1$. Thus there exists $(m,\alpha)$ such that
$0< \mu(X(U_{m,\alpha}))<\delta$ (note that $\delta>0$ is uniform in $(m,\alpha)$), and we obtain a contradiction between (\ref{mix}) and 
Lemma~\ref{intersection-of-cylinderSet}.
\end{proof}

 
Recall that $\Kk(\Tk) = \{ x \in \R^d : \Tk -x = \Tk \}$ is the group of periods of a tiling $\Tk$.


\begin{theorem}  \label{eigen-thm}
Let $\Tk$ be a  primitive  substitution tiling with expansion map $Q$ and tile-substitution $\om$. 

{\rm (i)}
If $\alpha \in \R^d$ is an eigenvalue of the measure-preserving system 
$(X_{\Tk}, \R^d, \mu)$ then
\begin{eqnarray} \label{eigenvalue-formula}
\lim_{n \to \infty} e^{2\pi i \langle Q^n z, \alpha \rangle} = 1  \ \ \ \mbox{for all $z \in \Xi_{\rm legal}(\Tk)$},  
\end{eqnarray}
and 
\be \label{period-cond}
e^{2\pi i \langle g, \alpha \rangle} = 1 \ \ \ \mbox{for all $g \in \mathcal{K}$}.
\ee

{\rm (ii)} Suppose, in addition, that $\Tk$ is repetitive, the tile-substitution $\om$ is recognizable modulo periods, and all the eigenvalues of $Q$ are algebraic integers. If (\ref{eigenvalue-formula}) and (\ref{period-cond}) hold, then $\alpha$ is an eigenvalue, and the eigenfunction can be chosen continuous.
\end{theorem}

The following is an immediate consequence of the theorem.

\begin{cor} \label{cor-eigen}
Suppose that
$\Tk$ is a repetitive primitive  non-periodic substitution tiling with  recognizable tile-substitution and linear expansion map $Q$, all of whose eigenvalues are algebraic integers. Then $\alpha\in \R^d$ is an eigenvalue of the
measure-preserving system $(X_\Tk,\R^d,\mu)$ if and only if the property (\ref{eigenvalue-formula}) holds. 
\end{cor} 


\medskip

\noindent {\bf Question.} {\em Let $Q$ be the expansion map of a substitution tiling with a finite set of prototiles up to translation. Is it necessarily true that all the eigenvalues of $Q$ are algebraic integers?
}

\medskip

In the FLC case the answer is positive, and the proof is not hard \cite{Thur} (see also \cite[Cor.\,4.2]{LeeSol:08}). Observe that even without FLC we have that $|\det(Q)|$ is an algebraic integer (a Perron number in the primitive substitution case),
since it is the dominant eigenvalue of the substitution matrix, corresponding to the left eigenvector whose components are the volumes of the prototiles. Thus, in particular, if $Q$ is a pure dilation $Q(x) = \theta x$, then $\theta$ is necessarily an algebraic integer.

\medskip

Before proceeding with the proof of Theorem~\ref{eigen-thm}, we need the following.

\begin{defi}
{\em A set of
algebraic integers $\Theta = \{\theta_1, \cdots, \theta_r \}$ is a
{\em Pisot family} if for any $1 \le j \le r$, every Galois
conjugate $\gamma$ of $\theta_j$, with $|\gamma| \ge 1$, is
contained in $\Theta$. For $r=1$, with $\theta_1$ real and $|\theta_1|>1$, this reduces to $|\theta_1|$ being a real Pisot number, and for $r=2$, with $\theta_1$ non-real and $|\theta_1|>1$, to $\theta_1$ being a complex Pisot number. 
Following \cite{Robi.lec}, we say that a family of algebraic integers is {\em totally non-Pisot} if it does not have a subset which is a Pisot family.
}
\end{defi}

In the course of the proof of Theorem~\ref{eigen-thm}, we will prove the following.

\begin{cor} \label{cor-Pisot-fam}
Let $\Tk$ be a primitive  substitution tiling with expansion map $Q$, whose eigenvalues are algebraic integers.

{\rm (i)} Suppose that $\alpha\in \R^d$ is an eigenvalue of the measure-preserving system $(X_\Tk,\R^d,\mu)$. Let $\Theta=\{\theta_1,\ldots,\theta_r\}$ be the set of eigenvalues of $Q$  (real and complex) such that the corresponding eigenvectors satisfy
$$
\langle \vec{e}_j,\alpha\rangle \ne 0.
$$
Then $\Theta$ is a Pisot family.

{\rm (ii)} Suppose that the set of eigenvalues of $Q$ is totally non-Pisot. Then $(X_\Tk,\R^d,\mu)$ is weakly mixing.
\end{cor}

In the case of $Q$ being a pure dilation we obtain the following.

\begin{cor} \label{cor-Pisot-dilate}
Let $\Tk$ be a primitive substitution tiling with expansion map $Q(x) = \theta x$, with $|\theta|>1$.

{\rm (i)} If the system $(X_\Tk,\R^d,\mu)$ is not weakly mixing (i.e., there exists a non-trivial eigenvalue), then $|\theta|$ is a Pisot number.

{\rm (ii)} If, in addition, $\Tk$ is  repetitive and the substitution is recognizable, then every measure-theoretic eigenvalue is also a topological eigenvalue, i.e., every eigenfunction may be chosen continuous.
\end{cor}

\begin{proof}[Proof sketch of Theorem~\ref{eigen-thm}] (i) 
For the necessity of (\ref{eigenvalue-formula}), the argument is similar  to the proof of \cite[Thm. 4.3]{soltil}, based on Lemma \ref{intersection-of-cylinderSet}. 
We use that any  $f\in L^2(X_\Tk,\mu)$ may be approximated by a linear combination of characteristic functions of the cylinder sets $X(U_{m,\alpha})$. The necessity of (\ref{period-cond}) is straightforward, proved as in \cite[\S4]{sol-eigen}.

(ii)
The argument is similar to the proof of \cite[Theorem 3.13]{sol-eigen}, but since the latter relied on FLC in several places, we will sketch it in more detail. Suppose that (\ref{eigenvalue-formula}) and
(\ref{period-cond}) hold, and define
$$
f_\alpha(\Tk-x) = e^{2\pi i \langle x,\alpha \rangle},\ \ x\in \R^d.
$$
The orbit $\{\Tk-x:\ x\in \R^d\}$ is dense in $X_\Tk$ by definition. It suffices to show that $f_\alpha$ is uniformly continuous on this orbit; then we can extend $f_\alpha$ to $X_\Tk$ by continuity, and this
extension will satisfy the eigenvalue equation. 

 The idea is, roughly, as follows. Suppose that $\Tk-x$ is very close to $\Tk-y$ in the tiling metric. Without loss of generality, we can assume that they agree exactly on a large neighborhood of the origin, say, $B_R(0)$, with $R\gg 1$. Since $\Tk-x=\om(\Tk-Q^{-1}x)$ and $\Tk-y = \om(\Tk- Q^{-1}y)$,  it follows from recognizability modulo periods (see Definition~\ref{def-recog}) 
that $\Tk-Q^{-1} x$ agrees with $\Tk-Q^{-1}y-Q^{-1}g_1$ for some period $g_1$, on a smaller, but still large neighborhood. We can repeat this argument a number of times, say $n$, depending on how large $R$ is, until we can only say that the resulting tilings:
$$
\Tk-Q^{-n}x\ \ \mbox{and}\ \ \Tk-Q^{-n}y - Q^{-n}g_1 - Q^{-n+1}g_2 - \cdots - Q^{-1}g_n
$$
agree on at least one tile near the origin, where $g_1,\ldots,g_n$ are translational periods of $\Tk$. But this means that
$$
Q^{-n}x - Q^{-n}y - Q^{-n}g_1 - Q^{-n+1}g_2 - \cdots - Q^{-1}g_n=:z\in \Xi(\Tk).
$$
Repetitivity implies that $\Xi(\Tk)=\Xi_{\rm legal}(\Tk)$, so $z\in \Xi_{\rm legal}(\Tk)$, and we have
$$
\langle x-y, \alpha\rangle = \langle Q^n z,\alpha\rangle \ \ (\mbox{mod}\ \Z),
$$

\noindent using condition (\ref{period-cond}) and the fact that the group of periods is mapped by $Q$ into itself. From (\ref{eigenvalue-formula}), we have that $e^{2 \pi i \langle Q^n z,\alpha\rangle}\approx 1$ for large $n$, and hence $f_\alpha(\Tk-x) \approx f_\alpha(\Tk-y)$, as desired. This is, of course, far from a proof, and significant amount work is needed to realize this scheme.

The actual proof in \cite{sol-eigen} proceeded with four lemmas. The first one, \cite[Lemma 4.1]{sol-eigen} claimed that the eigenvalues of the expansion map are algebraic integers; now this is an assumption. Next we need the extension of \cite[Lemma 4.2]{sol-eigen}:

\begin{lemma} \label{lem-Pisot} Suppose that $Q$ is a linear expansion map whose eigenvalues are algebraic integers. Then there exists $\rho \in (0,1)$, depending only on $Q$, such that, if
 (\ref{eigenvalue-formula}) holds for  $z\in \Xi(\Tk)$, then
 \be\label{conv-exp}
 \left|e^{2\pi i \langle Q^nz, \alpha\rangle} -1\right| < C_z\rho^n,\ \ n\in \N.
 \ee
 \end{lemma}
\begin{proof}
Denote by $\|t\|_{\R/\Z}$ the distance from $t\in \R$ to the nearest integer. Decompose $z$ into a linear combination of eigenvectors and root vectors (if any) of the expansion map $Q$. Then  (\ref{eigenvalue-formula})  becomes
\be \label{eq-Dioph}
\lim_{n\to \infty} \left\|\sum_{i=1}^r P_i(n) \theta_i^n \right\|_{\R/\Z} = 0,
\ee
where $P_i(n)$ are non-zero polynomials and $\theta_1,\ldots,\theta_r$ are all the eigenvalues of $Q$ (real and complex) for which $z$ has a non-zero coefficient and for which there is an eigenvector $\vec{e}_i$ satisfying $\langle {\vec{e}_i},\alpha\rangle\ne 0$. The degree of the polynomial $P_i$ is the maximal size of a Jordan block for $\theta_i$ which contributes non-trivially to the decomposition.
We can now apply a generalization of the classical Pisot's Theorem, due to K\"{o}rnei \cite[Theorem 1]{Korn} and/or a similar result of Mauduit \cite{Mauduit}, which asserts that if (\ref{eq-Dioph}) holds, then $\Theta=\{\theta_1,\ldots,\theta_r\}$ is a Pisot family and (\ref{conv-exp}) is satisfied, with any $\rho\in (0,1)$ that is larger in absolute value than all the Galois conjugates of $\theta_j$'s from $\Theta$ that do not appear in $\Theta$.
\end{proof}

The rest of the proof of Theorem~\ref{eigen-thm} proceeds exactly as in \cite{sol-eigen}, after noting that \cite[Lem.\ 4.5]{LeeSol:08} was obtained without the FLC assumption. The repetitivity was needed there in the same place as in the rough scheme above, namely, to claim that $\Xi(\Tk)=\Xi_{\rm legal}(\Tk)$.
\end{proof}

\begin{proof}[Proof of Corollary~\ref{cor-Pisot-fam}]
(i) This is deduced from  Theorem~\ref{eigen-thm}(i), following the argument in the proof of Lemma~\ref{lem-Pisot}. The only additional observation needed is that the set of return vectors $\Xi(\Tk)$ is relatively dense in $\R^d$,
hence if $\langle \vec{e}_j,\alpha\rangle \ne 0$, we can find $z\in \Xi(\Tk)$ whose decomposition into a linear combination of eigen- and root vectors of $Q$ will have a non-zero coefficient with respect to 
$\vec{e}_j$.

(ii) This follows from part (i) by the definition of totally non-Pisot family.
\end{proof}

\begin{proof}[Proof of Corollary~\ref{cor-Pisot-dilate}]
The first claim is immediate from Corollary~\ref{cor-Pisot-fam}. The second claim follows from Theorem~\ref{eigen-thm}(i) and the remark about $|\det(Q)| = |\theta|^d$ being an algebraic integer.
\end{proof}


\section{Rigidity for substitution tilings. Examples.} \label{section:Rigidity}

There are many notions of ``rigidity'' in mathematics. Ours originates from the work of Kenyon \cite{Kenyon.rigidity,Kenyon.inflate}. In \cite[Theorem 1]{Kenyon.rigidity} it is shown that any sufficiently small perturbation of a planar tiling with finitely many prototiles (not necessarily a substitution tiling) must have an ``earthquake'', or ``fault line'' discontinuity (we refer to \cite{Kenyon.rigidity}, as well as to \cite[Section 5]{Frank}, for details). In \cite[Cor.\ 3]{Kenyon.rigidity} Kenyon derives from this that a tiling of $\R^2$ with a finite number of prototiles, up to translation, either has FLC, or the union of tile boundaries contains arbitrarily long line segments. In the planar FLC case, with the expansion $Q$ given by a complex multiplication $z\mapsto \lam z$, with $\C\cong \R^2$, the impossibility of a small perturbation led
Kenyon \cite[Section 5, p.\ 484]{Kenyon.GAFA} to conclude, with a reference to the method of \cite{Kenyon.inflate}, that for such tilings holds the inclusion $\Xi(\Tk) \subset \Z[\lambda] \xi$, for some
$\xi\in \C$. In the paper \cite{LeeSol:12} we gave a careful proof of this, as well as a generalization to the case of $\R^d$, under some algebraic assumptions, as stated below.
We start with a definition.

\begin{defi} 
{\em
Let $d=mJ$, for some $J\ge 1$, and let $\Tk $ be a  substitution tiling in $\R^d$ with expansion map $Q$.  We represent  $\R^d=\bigoplus_{j=1}^J H_j$, where
$$
H_{j} = \{0\}^{(j-1)m} \times \R^{m} \times\{0\}^{d- jm}\,.
$$
Assume that $Q$ is diagonalizable over $\C$, and all the eigenvalues of $Q$ are algebraic conjugates with the same multiplicity $J$. The tiling $\Tk$ is said to be {\em rigid} if there exists a linear  isomorphism $\rho: \R^d \to \R^d$ such that 
\be \label{rigidity-property}
\rho Q = Q \rho \ \ \  \mbox{and} \ \ \Xi(\Tk) \subset \rho(\Z[Q]\balpha_1 + \cdots + \Z[Q]\balpha_J), 
\ee
where  $\balpha_j \in H_{j}$, $1 \le j \le J$, are such that for each $1 \le n \le d$, 
\[ \label{def-alpha}
 (\balpha_{j})_{n} =
\left\{\begin{array}{ll}
                             1 \ \ \ & \mbox{if} \ \  (j-1)m + 1 \le n \le jm;\\
                             0 \ \ \ & \mbox{else} \,.
                           \end{array} \right.
\]
Here $\Z[Q]:= \{\sum_{n=0}^N a_n Q^n:\ a_n \in \Z,\ N\in \N\}$. 

The most basic case is when $Q(x) = \theta x$ is a pure dilation in $\R^d$. Then $\Tk$ is rigid if and only if there exists a basis $\{x_1,\ldots,x_d\}$ of $\R^d$ such that
\[\Xi(\Tk) \subset \Z[\theta]x_1 + \cdots + \Z[\theta] x_d.\]

We say that a representable substitution Delone $\kappa$-set $\Lb$ is rigid, if the corresponding tiling $\Tk$ is rigid, that is, (\ref{rigidity-property}) holds for $\Xi(\Lb)$ instead of $\Xi(\Tk)$.

}
\end{defi}
 
\begin{theorem} \cite[Theorem 4.1]{LeeSol:12}
Let $\Tk $ be a repetitive primitive substitution tiling with expansion map $Q$. Suppose that $\Tk$ has FLC, $Q$ is diagonalizable, and all the eigenvalues of $Q$ are algebraic conjugates with the same multiplicity $J$. Then
$\Tk$ is rigid.
\end{theorem}

In the special case when $Q(x) = \theta x$ this was proved in \cite[\S 5]{sol-eigen}.


Substitution tilings (with the assumptions on $Q$ stated above) with FLC have the rigidity property, but rigidity does not imply FLC.
The following substitution tiling by Frank and Robinson  \cite{FraRob}  demonstrates this. (See also \cite[Ex.\ 7.5]{ABBLS}). 

\begin{example}[\cite{FraRob}] \label{FraRob-example}
\begin{figure}[ht]  \centering
\includegraphics[width=0.8\textwidth]{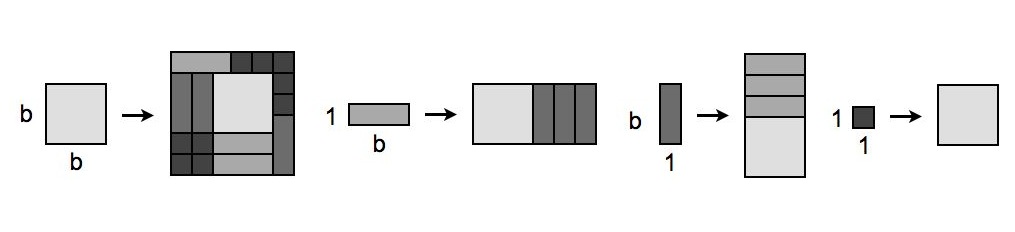}
\caption{Prototiles of the Frank-Robinson substitution tiling without FLC}
\end{figure}
Take
{\footnotesize
\begin{eqnarray*}
Q A_1 &=& (A_1 +(2,2))  \cup (A_2 +(2,0))  \cup (A_2 +(2,1)) \cup (A_2 +(0,b+2)) \\
&& \cup (A_3 +(0,2)) \cup (A_3 +(1,2)) \cup (A_3 +(b+2,0))   \cup A_4 \cup (A_4 +(1,0))  \\
&& \cup (A_4 +(0,1)) \cup (A_4 +(1,1)) \cup (A_4 +(b+2,b)) \cup (A_4 +(b+2,b+1))  \\
&& \cup (A_4 +(b+2,b+2))  \cup (A_4 +(b+1,b+2))  \cup (A_4 +(b,b+2))  \\
Q A_2 &=& A_1 \cup (A_3 + (b, 0)) \cup (A_3 + (b+1, 0)) \cup (A_3 + (b+2, 0))  \\
Q A_3 &=& A_1 \cup (A_2 + (0, b)) \cup (A_2 + (0, b+1)) \cup (A_2 + (0, b+2))  \\
Q A_4 & = & A_1 \,,
\end{eqnarray*} }
where $b$ is the largest root of $x^2 - x -3 = 0$ and $Q = \left( \begin{array}{cc}
                              b & 0 \\
                              0 & b
                              \end{array}
                     \right)$. Note that $b$ is not a Pisot number. This defines a tile-substitution, which has a fixed point $\Tk$ of infinite local complexity.
One sees that each set of translation vectors satisfies $\mathcal{D}_{ij} \subset \Z[Q](1, 0) + \Z[Q](0, 1)$. 
Hence 
\[
\Xi(\Tk) \subset \Z[Q](1, 0) + \Z[Q](0, 1),
\]
whence the rigidity property holds. By Corollary~\ref{cor-Pisot-dilate}(i), the dynamical system $(X_\Tk,\R^d,\mu)$ is weakly mixing.
\end{example}

Although the definition of rigidity seems to be more complicated than the notion of FLC, it is easier to check than FLC.  One only needs to consider the  sets of translation vectors $\Dk_{ij}$ and find the smallest module over $\Z[Q]$ containing them.

\medskip

{On the other hand, there are examples of substitution tilings for which  neither FLC nor the rigidity hold. 
In this case, unlike \cite[Thm.\ 5.2]{LeeSol:12}, we cannot expect the equivalence between relative dense set of eigenvalues of the dynamical system and being not weakly mixing.
}

\begin{example}[\cite{Kenyon92}] \label{ex-kenyon}
Consider the substitution tiling $\Tk$ in $\R^2$ with a single prototile $T$ and expansion
$Q = \left( \begin{array}{cc}
                              3 & 0 \\
                              0 & 3
                              \end{array}
                     \right)$ such that 
                     \[  Q T = \bigcup_{d \in \mathcal{D}} (T + d)\]
                     where \[ \mathcal{D} = \{(0,-1), (0,0), (0,1), (-1, -1), (-1, 0), (-1, 1), (1, -1+a), (1, a), (1, 1+a)\}\] and $a\in\R$ is irrational.
Note that \[ \Xi(\Tk) \subset \Z[Q](1,0) + \Z[Q](0,1) + \Z[Q](0, a)\]
and $\Z[Q](1,0) + \Z[Q](0,1) + \Z[Q](0, a)$ is the minimal module over $\Z[Q]$ containing $\Xi(\Tk)$.  
Thus $\Tk$ does not have the rigidity property. We observe that the tiling is periodic in the direction of $y$-axis, so the tiling dynamical system has non-trivial eigenfunctions and is not weakly mixing.

On the other hand, let $\alpha=(\alpha_1,\alpha_2)\in \R^2$ be an eigenvalue. By the condition (\ref{eigenvalue-formula}), {since $(1,0), (0,1) \in \Xi_{\rm legal}(\Tk)$, we have that $3^n \alpha_1\to 0$ mod $1$ and $3^n \alpha_2\to 0$ mod $1$, hence $\alpha_1$ and $\alpha_2$ are 3-adic rationals. On the other hand, $(1,a)\in \Xi_{\rm legal}(\Tk)$ as well, hence $3^n \alpha_1 + 3^n a\alpha_2\to 0$ mod 1, and we conclude that $a\alpha_2$ is a 3-adic rational. But since $a$ is irrational, we obtain that $\alpha_2=0$.}
\end{example}
\begin{figure}[ht]  
\centering
\includegraphics[width=0.3\textwidth]{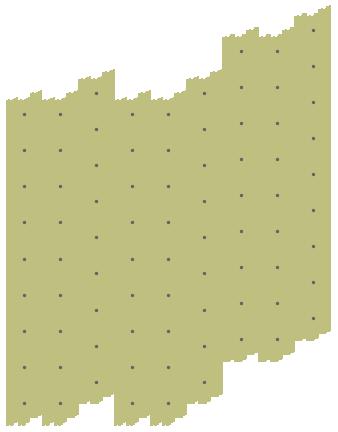}
\caption{A patch of the tiling from Example~\ref{ex-kenyon} for $a = 2-\sqrt{2}$. {The dots in the figure indicate the representative points of tiles.} }
\end{figure}


See also \cite[Chap.1]{BaakeGrimm2} for a figure of the Kenyon's example and references to other related examples. 

\medskip

We can modify the Kenyon's example to make it non-periodic. First take a constant length substitution tiling and a Fibonacci substitution tiling in $\R$ and consider a direct product substitution, then slide the last column relative to the first column. We give a precise example of this kind below.

\begin{example} \label{ex-kenyon.modif}
Consider an expansion map
$Q = \left( \begin{array}{cc}
                              3 & 0 \\
                              0 & \tau
                              \end{array}
                     \right)$, where $\tau$ is the golden ratio, $\tau^2-\tau-1=0$.
The tile equations are 
{\footnotesize
\begin{eqnarray*}
Q A_1 &=& A_1  \cup (A_1 +(\tau,0))  \cup (A_1 +(2\tau,a)) \cup (A_2 +(0, \tau)) 
\cup   (A_2 + (\tau,\tau)) \cup (A_2 + (2\tau, a + \tau)) \\
Q A_2 &=& A_1 \cup (A_1 + (\tau,0))    \cup (A_1 + (2\tau,a)),
\end{eqnarray*} }
where $a\in\R$ is irrational, such that $a\not\in \Q(\tau)$.
\begin{figure}[ht]  
\centering
\includegraphics[width=0.6\textwidth]{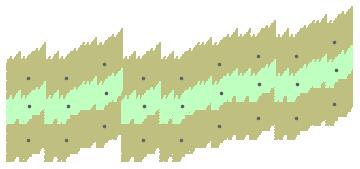}
\caption{Modification of Kenyon's example. The figure shows  a patch of the substitution tiling in the case of $a =  2-\sqrt{2}$. 
{The dots in the figure indicate the representative points of tiles.}  }
\end{figure}
Note that \[ \mathcal{D}_{ij} \subset (\tau,0)\Z[Q] +(0,\tau)\Z[Q] + (2\tau,a)\Z[Q] \]
and $(\tau,0)\Z[Q] +(0,\tau)\Z[Q] + (2\tau,a)\Z[Q]$ is the minimal module over $\Z[Q]$ containing $\Xi(\Tk)$. {Since 
\[ (\tau,0), (0,\tau), (2\tau,a) \] cannot be linearly independent over $\R$},  the tiling $\Tk$ does not have the rigidity property. We observe that the tiling is non-periodic. However 
note that $\lim_{n \to \infty} \langle (\tau,0), (x,y)Q^n \rangle = 0 $ mod $\Z$ for any $(x, y) \in \Xi(\Tk)$. Moreover, one can check by inspection that it is repetitive and recognizable.
So the tiling dynamical system has non-trivial eigenvalue and is not weakly mixing.
\end{example}

 In \cite{LeeSol:12}, we showed under the rigidity assumption the equivalence between the  Pisot family property and the existence of a relatively dense set of eigenvalues for $(X_{\Tk}, \R^d, \mu)$. However, we note that the equivalence can be proved without assuming FLC.  So we revisit the theorem under the assumption of the rigidity. 

Let us first recall the notion of a Meyer set. A Delone set $Y\subset \R^d$ is {\em Meyer} if it is
relatively dense and $Y-Y$ is uniformly discrete.

{Substitution tilings with the Meyer property always have FLC. But the converse is not true.
Substitution tilings in $\R$ always have FLC, but not the Meyer property. An example in \cite[Example 7.8]{soltil} shows such a substitution tiling in $\R^2$. The example is a special case of Kenyon's construction \cite{Kenyon.GAFA} where $\lambda$ is not a complex Pisot number.  (We should correct that example slightly, as follows:
 $\theta(a)=b, \theta(b)=c, \theta(c)=a^{-3}b^{-1}$ and $\lambda$ is the non-real root of the equation $\lambda^3 + \lambda +3 = 0$.)
}

\medskip

Lagarias \cite[Problem 4.11]{Lag00} asked whether a primitive substitution Delone set, which is pure-point diffractive, is necessarily Meyer. In \cite{LeeSol:08} we gave a positive answer
under the FLC assumption. Here we show that the FLC assumption may be dropped if we assume  that all the eigenvalues of the expansion map are algebraic integers.
For the notion of diffraction spectrum and its connection with dynamical spectrum we refer the reader to \cite{BaaLen17} and references therein. Briefly, pure point diffraction implies that there is a relatively dense set of Bragg peaks, which in turn implies that there is a relatively dense set of eigenvalues. Now the next proposition provides the answer.

\begin{prop} \label{prop-Lag-question}
Let $\Lb$ be a repetitive primitive substitution Delone $\kappa$-set in $\R^d$ with expansion map $Q$. Suppose that all the eigenvalues of $Q$ are algebraic integers. Assume that the 
set of eigenvalues of $(X_{\Lbs}, \R^d, \mu)$ is relatively dense. Then $\supp(\Lb)$ is a Meyer set.
\end{prop}

\begin{proof}
Under the assumption that all the eigenvalues of $Q$ are algebraic integers, the only place that requires FLC in the proof of \cite[Prop. 4.8]{LeeSol:08} is \cite[Cor. 4.7]{LeeSol:08}. But we obtained Theorem \ref{eigen-thm} to replace \cite[Cor. 4.7]{LeeSol:08}, without assuming FLC. So the result follows. 
\end{proof}

On the other hand, we should mention that the answer to Lagarias \cite[Problem 4.10]{Lag00} is negative: there exists a repetitive pure point diffractive set, which is not Meyer. An example (the ``scrambled Fibonacci tiling'') was constructed by Frank and Sadun \cite{FraSa:fusion}.

\begin{theorem}  \label{PisotFamily-MeyerSet}
Let $\Lb$ be a repetitive primitive substitution Delone $\kappa$-set $\R^d$ with expansion map $Q$. Suppose that the corresponding tile-substitution tiling $\om$ is recognizable modulo periods, $Q$ is diagonalizable, and all the eigenvalues of $Q$
are algebraic conjugates with the same multiplicity. Then the
following are equivalent:
\begin{itemize}
\item[(i)] the set of eigenvalues of $(X_{\Lbs}, \R^d, \mu)$
    is relatively dense;
\item[(ii)] $\supp(\Lb)$ is a Meyer set;
\item[(iii)] $\Lb$ is rigid and the set of eigenvalues of $Q$ forms a Pisot family. 
\end{itemize}
\end{theorem}

\begin{proof} 
(i) $\Rightarrow$ (ii) This is a special case of Proposition~\ref{prop-Lag-question}. \\
(ii) $\Rightarrow$ (iii) If $\supp(\Lb)$ is a Meyer set, then $\Lb$ has  FLC. So from \cite[Theorem 4.1 and Prop. 1]{LeeSol:12}, the result follows.\\
(iii) $\Rightarrow$ (i) We note that \cite[Prop. 1]{LeeSol:12} 
does not use the assumption of FLC. Instead it requires the rigidity of $\Xi(\Tk)= \Xi(\Lb)$. So from the assumption the result follows. 
\end{proof}

So, in particular, we obtain that a relatively dense set of eigenvalues for the dynamical system $(X_{\Lbs}, \R^d, \mu)$ is impossible in the case of infinite local complexity.

\medskip

As we saw from Kenyon's example, weak mixing is not equivalent to the Pisot family condition in general. However, we recover this if we assume rigidity.

\begin{cor}   \label{PisotFamily-weaklymixing-MeyerSet}
Let $\Lb$ be a  repetitive primitive substitution Delone $\kappa$-set $\R^d$ with expansion map $Q$. Suppose that the corresponding tile-substitution $\om$ is recognizable modulo periods, $Q$ is diagonalizable over $\C$, and all the eigenvalues of $Q$
are algebraic conjugates with the same multiplicity.
Assume that $\Lb$ is rigid.
Then the
following are equivalent:
\begin{itemize}
\item[(i)] the set of eigenvalues of $(X_{\Lbs}, \R^d, \mu)$
is relatively dense;
\item[(ii)] $(X_{\Lbs}, \R^d, \mu)$ is not weakly mixing;
\item[(iii)] the set of eigenvalues of $Q$ forms a Pisot family;
\item[(iv)] $\supp(\Lb)$ is a Meyer
set.
\end{itemize}
\end{cor}

\begin{proof} 
(i) $\Rightarrow$ (ii) is obvious.\\
(ii) $\Rightarrow$ (iii) follows by \cite[Lemma 5.1]{LeeSol:12}.\\
(iii)  $\Rightarrow$ (iv) and (iv)  $\Rightarrow$ (i) hold by Theorem \ref{PisotFamily-MeyerSet}. 
\end{proof}


\section{Appendix}

Recall that for Delone $\kappa$-sets $\Lb_1, \Lb_2 \in X$, we let
\[ d(\Lb_1, \Lb_2): = \mbox{min}\{\widetilde{d}(\Lb_1, \Lb_2), 2^{-\frac{1}{2}}\},\]
where
$$
\widetilde{d}(\Lb_1, \Lb_2) = \mbox{inf} \Bigl\{ \epsilon>0\ | \ B_{\frac{1}{\epsilon}}(0) \cap \Lb_2 \subset \Lb_1 + B_{\epsilon}(0) \  \mbox{and}
  \ B_{\frac{1}{\epsilon}}(0) \cap \Lb_1  \subset \Lb_2 + B_{\epsilon}(0)  \Bigr\}. 
$$
Here we show that this is a metric. The only issue is the triangle inequality.
Suppose that $d(\Lb_1, \Lb_2) \le \epsilon_1$, $d(\Lb_2, \Lb_3) \le \epsilon_2$ and $\epsilon_2, \epsilon_3 < 2^{-\frac{1}{2}}$.
Then 
\begin{eqnarray*}
B_{\frac{1}{\epsilon_1}}(0) \cap \Lb_2 \subset \Lb_1 + B_{\epsilon_1}(0) , \\
B_{\frac{1}{\epsilon_1}}(0) \cap \Lb_1 \subset \Lb_2 + B_{\epsilon_1}(0), \\
B_{\frac{1}{\epsilon_2}}(0) \cap \Lb_3 \subset \Lb_2 + B_{\epsilon_2}(0), \\
B_{\frac{1}{\epsilon_2}}(0) \cap \Lb_2 \subset \Lb_3 + B_{\epsilon_2}(0)  \,.
\end{eqnarray*}
We want to show that 
\begin{eqnarray*}
B_{\frac{1}{\epsilon_1+\epsilon_2}}(0) \cap \Lb_3 \subset \Lb_1 + B_{\epsilon_1+ \epsilon_2}(0), \\
 B_{\frac{1}{\epsilon_1 + \epsilon_2}}(0) \cap \Lb_1 \subset \Lb_3 + B_{\epsilon_1+ \epsilon_2}(0)  \,.
\end{eqnarray*}
For any $x \in B_{\frac{1}{\epsilon_1+ \epsilon_2}}(0) \cap \Lb_3 $, there exists $y \in \Lb_2$ such that $y+t = x$ and $t \in B_{\epsilon_2}(0)$.
Then 
\[ B_{\frac{1}{\epsilon_1}}(t) \cap (\Lb_2 + t) \subset \Lb_1 + t + B_{\epsilon_1}(0). \]
Note that 
\[ B_{\frac{1}{\epsilon_1}}(t)  \supset B_{\frac{1}{\epsilon_1} - \epsilon_2}(0) \  \mbox{and}  \  \frac{1}{\epsilon_1} - \epsilon_2 > \frac{1}{\epsilon_1 + \epsilon_2}. \]
So 
\[x \in B_{\frac{1}{\epsilon_1+\epsilon_2}}(0) \cap (\Lb_2 + t) \subset \Lb_1 + B_{\epsilon_1+ \epsilon_2}(0) \,. \]
Therefore 
\[ B_{\frac{1}{\epsilon_1+\epsilon_2}}(0) \cap \Lb_3 \subset \Lb_1 + B_{\epsilon_1+ \epsilon_2}(0)  \,. \]
Similarly,
\[ B_{\frac{1}{\epsilon_1+\epsilon_2}}(0) \cap \Lb_1 \subset \Lb_3 + B_{\epsilon_1+ \epsilon_2}(0) \,. \]
Thus $d(\Lb_1, \Lb_3) \le \epsilon_1 + \epsilon_2$.

\section{Acknowledgement}

We would like to thank to M. Baake, D. Frettl\"{o}h, C. Richard,  and N. P. Frank for valuable comments and discussions. We are grateful to the anonymous referee for many constructive suggestions which helped to improve the exposition. J.-Y. Lee would like to acknowledge the support by Local University Excellent Researcher Supporting Project through the Ministry of Education of the Republic of Korea and National Research Foundation of Korea (NRF) (2017078374). She is also grateful for the support of the
Korea Institute for Advanced Study (KIAS). The research of B. Solomyak was supported in part by the Israel Science Foundation (Grant 396/15). He is grateful to the KIAS, where part of this work was done, for hospitality.

\end{document}